\documentclass[12pt]{amsart}
\usepackage{latexsym,fancyhdr,amssymb,color,amsmath,amsthm,graphicx,listings,comment}
\usepackage[section]{placeins}
\newtheorem{thm}{Theorem} \newtheorem{lemma}{Lemma}  \newtheorem{coro}{Corollary}
\definecolor{red1}{rgb}{1,0.9,0.9} \definecolor{blue1}{rgb}{0.9,0.9,1} \definecolor{green1}{rgb}{0.9,1,0.9} 
\definecolor{yellow1}{rgb}{1,1,0.9} \definecolor{yellow2}{rgb}{1,1,0.8}

\def\question#1{ \vspace{2mm} \begin{center} \parbox{11.2cm}{{\bf Question:} #1} \vspace{2mm} \end{center} }

\let\paragraph\subsection

\def\B#1#2{{#1\choose #2}}

\title{On a Dehn-Sommerville functional for simplicial complexes}
\author{Oliver Knill}
\date{May 29, 2017}
\address{Department of Mathematics \\ Harvard University \\ Cambridge, MA, 02138 }

\begin{document}
\maketitle

\begin{abstract}
Assume $G$ is a finite abstract simplicial complex with $f$-vector $(v_0,v_1, \dots)$,
and generating function $f(x) = \sum_{k=1}^{\infty} v_{k-1} x^k=v_0 x + v_1 x^2+ v_2 x^3 + \cdots$,
the Euler characteristic of $G$ can be written as $\chi(G)=f(0)-f(-1)$. We study here
the functional $f_1'(0)-f_1'(-1)$, where $f_1'$ is the derivative of the 
generating function $f_1$ of $G_1$. The Barycentric refinement $G_1$
of $G$ is the Whitney complex of the finite simple graph for which the faces of $G$ 
are the vertices and where two faces are connected if one is a subset of the other. 
Let $L$ is the connection Laplacian of $G$, which is $L=1+A$, where $A$ is the adjacency 
matrix of the connection graph $G'$, which has the same vertex set than $G_1$
but where two faces are connected they intersect. We have $f_1'(0)={\rm tr}(L)$ 
and for the Green function $g=L^{-1}$ also $f_1'(-1)={\rm tr}(g)$ so that 
$\eta_1(G) = f_1'(0)-f_1'(-1)$ is equal to $\eta(G)={\rm tr}(L-L^{-1})$. The 
established formula 
${\rm tr}(g)=f_1'(-1)$ for the generating function of $G_1$ complements the determinant 
expression ${\rm det}(L)={\rm det}(g)=\zeta(-1)$ for the Bowen-Lanford zeta function 
$\zeta(z)=1/{\rm det}(1-z A)$ of the connection graph $G'$ of $G$. We also establish a
Gauss-Bonnet formula $\eta_1(G) = \sum_{x \in V(G_1)} \chi(S(x))$, 
where $S(x)$ is the unit sphere of $x$ the graph generated by all vertices 
in $G_1$ directly connected to $x$. Finally, we point out that
the functional $\eta_0(G) = \sum_{x \in V(G)} \chi(S(x))$ on graphs 
takes arbitrary small and arbitrary large values on every homotopy type of graphs. 
\end{abstract}

\section{Setup}

\paragraph{}
A {\bf finite abstract simplicial complex} $G$ is a finite set of non-empty sets with 
the property that any non-empty subset of a set in $G$ is in $G$. The elements in $G$ are called 
{\bf faces} or {\bf simplices}. Every such complex defines two finite simple graphs $G_1$ and $G'$, 
which both have the same vertex set $V(G_1)=V(G')=G$. For the graph $G_1$, two vertices are 
connected if one is a subset of the other; in the graph $G'$, two faces are connected, if they intersect. 
The graph $G_1$ is called the {\bf Barycentric refinement} of $G$; 
the graph $G'$ is the {\bf connection graph} of $G$. The graph $G_1$ is a subgraph of $G'$ which shares
the same topological features of $G$. On the other hand, the connection graph is fatter and
be of different topological type: already the Euler characteristic $\chi(G)$ and $\chi(G')$ can differ. 
Both graphs $G_1$ and $G$ are interesting on their own but they are linked in various ways as we hope to
illustrate here. Terminology in this area of combinatorics is rich. One could stay within simplicial complexes
for example and deal with ``flag complexes", complexes which is a Whitney complex of its $1$-skeleton graphs.
The complexes $G_1$ and $G'$ are by definition of this type. We prefer in that case to use terminology of graph
theory. 

\paragraph{}
Let $A$ be the adjacency matrix of the connection graph $G'$.
Its Fredholm matrix $L=1+A$ is called the {\bf connection Laplacian} of $G$. We know that $L$ is 
unimodular \cite{Unimodularity} so that the {\bf Green function operator} $g=L^{-1}$ has integer entries.
This is the {\bf unimodularity theorem} \cite{Helmholtz}. The Bowen-Lanford zeta function of the graph $G'$ is
defined as $\zeta(s) = {\rm det}((1-sA)^{-1})$. As $\zeta(-1)$ is either $1$ or $-1$, we can see the determinant
of $L$ as the value of the zeta function at $s=-1$. 
We could call $H=L-L^{-1}$ the {\bf hydrogen operator} of $G$. The reason is that
classically, if $L=-\Delta$ is the Laplacian in $R^3$, then $L^{-1}$ is an integral operator with
entries $g(x,y) = 1/|x-y|$. Now, $H \psi(y) = (L \psi)(y) - \psi(y)/|x-y|$ is the Hamiltonian of a Hydrogen
atom located at $x$, so that $H$ is a sum of a kinetic and potential part, where the potential is determined
by the inverse of $L$. When replacing the multiplication operation with a convolution operation, then $L^{-1}$ 
takes the role of the potential energy. Anyway, we will see that the trace of $H$ is an 
interesting variational problem.

\paragraph{}
There are various variational problems in combinatorial topology or
in graph theory. For the later, see \cite{BollobasExtremal}. An example in polyhedral combinatorics 
is the upper bound theorem, which characterizes the maxima of the discrete volume among all 
convex polytopes of a given dimension and number of vertices \cite{Stanley1996}. 
An other example problem is to maximize the Betti number $b(G)=\sum_{i=0} b_i$
which is bounded below by $\chi(G)=\sum_{i=0} (-1)^i b_i $ which we know to grow 
exponentially in general in the number of elements in $G$ and for which upper bounds
are known too \cite{Adamaszek}.
We have looked at various variational problems in \cite{KnillFunctional} and at higher order 
Euler characteristics in \cite{valuation}. Besides extremizing functionals on geometries, one can 
also define functionals on the on the set of unit vectors of the Hilbert space $H^n$ generated
by the geometry. An example is the free energy $(\psi,L\psi) - T S(|\psi|^2)$ which uses 
also entropy $S$ and temperature variable $T$ \cite{Helmholtz}. 

\paragraph{}
Especially interesting are functionals which characterize geometries.
An example is a necessary and sufficient condition for a $f$-vector of a simplicial d-polytope
to be the $f$-vector of a simplicial complex polytope, conjectured 1971 and proven
in 1980 \cite{BilleraLee,Stanley1980}. Are there variational conditions which filter out
discrete manifolds? We mean with a discrete manifold a connected finite abstract 
simplicial complex $G$ for which every unit sphere $S(x)$ in $G_1$ is a sphere. The notion of 
sphere has been defined combinatorially in discrete Morse approaches using critical points 
\cite{forman95} or discrete homotopy \cite{I94a}. A $2$-complex for example is a discrete 
$2$-dimensional surface. In a 2-complex, we ask that every unit sphere in $G_1$ is a circular graph 
of length larger than $3$. For a 2-complex the $f$-vector of $G_1$ obviously satisfies 
$2 v_1-3v_2=0$ as we can count the number of edges twice by adding up 3 times the number of triangles. 
The relation $2v_1-3v_2=0$ is one of the simplest Dehn-Sommerville relations. 
It also can be seen as a zero curvature condition for 
$3$-graphs \cite{cherngaussbonnet} or then related to eigenvectors to Barycentric refinement 
operations \cite{valuation,KnillBarycentric2}. Dehn-Sommerville relations can be seen
as zero curvature conditions for Dehn-Sommerville invariants in a higher dimensional complex.

\paragraph{}
One can wonder for example whether a condition like $\eta(G) = 2v_1-3v_2=0$ for the $f$-vector $(v_0,v_1,v_2)$ 
of the Barycentric refinements $G_1$ of a general $2$-dimensional abstract finite simplicial complex $G$ 
forces the graph $G_1$ to have all unit spheres to be finite unions of circular graphs. 
For this particular functional, this is not the case. There are examples of discretizations of 
varieties with $1$-dimensional set of singular points for which $2v_1-3v_2$ is negative. 
An example is $C_n \times F_8$, the Cartesian product of
a circular graph with a figure $8$ graph. An other example is a $k$-fold suspension of a circle 
$G=C_n + P_k$, where $C_n$ is the circular graph, $P_k$ the $k$ vertex graph without edges and $+$ is the
Zykov join which takes the disjoint union of the graphs and connects additionally any two vertices from different
graphs. In that case, $\eta_0(G)=n(2-k)$ which is zero only in the discrete manifold case $k=2$
where we have a discrete 2-sphere, the suspension of a discrete circle. 

\paragraph{}
Our main result here links a spectral property with a combinatorial property. It builds on 
previous work on the connection operator $L$ and its inverse $g=L^{-1}$. 
We will see that $\eta(G_1)={\rm tr}(L-L^{-1})$, where $L$ is the connection 
Laplacian of $G$, which remarkably is always invertible. 
If $G$ has $n$ faces=simplices=sets in $G$, the matrix $L$ is a $n \times n$
matrix for which $L_{xy}=1$ if $x$ and $y$ intersect and where $L_{xy}=0$ if $x \cap y$ is empty. 
We establish that the {\bf combinatorial functional} $\eta(G_1)=2v_1-3v_2 +4v_3 -5 v_4 + \dots$
which is also the {\bf analytic functional} $f'(0)-f'(-1)$ for a generating function 
$f(x)=\sum_{k=1}^{\infty} v_{k-1} x^k$ is the same than the {\bf algebraic functional} 
${\rm tr}(L-L^{-1})$ and also equal to the 
{\bf geometric functional} $\eta_1(G) = \sum_{x \in V(G_1)} \chi(S(x))$. The later
is a Gauss-Bonnet formula which in general exists for any linear or multi-linear valuation 
\cite{valuation}. 

\paragraph{}
The functional $\eta_1$ is a valuation like the Euler characteristic $\chi(G) = v_0-v_1+v_2- \dots$
whose combinatorial definition can also be written as $f(0)-f(-1)$ or as a Gauss-Bonnet formula
$\sum_x K(x)$ or then as the super trace of a heat kernel ${\rm tr}(e^{-t L})$ by McKean-Singer
\cite{knillmckeansinger}. The Euler curvature $K(x)=\sum_{k=0}^{\infty} (-1)^k v_{k-1}(S(x))/(k+1)$
\cite{cherngaussbonnet} could now be written as $K(x)=F(0)-F(-1)$, where 
$F(t)=\int_0^t f(s) \; ds = \sum_{k=0}^{\infty} v_{k-1} x^{k+1}/(k+1)$
is the anti-derivative of the {\bf reduced generating function} $1-f$ of $S(x)$. We see a common theme
that $F(0)-F(-1), f(0)-f(-1), f'(0)-f'(-1)$ all appear to be interesting.

\paragraph{}
Euler characteristic is definitely the most fundamental valuation as it is related 
to the unique eigenvector of the eigenvalue $1$ of the Barycentric refinement operator. It also has
by Euler-Poincar\'e a cohomological description $b_0-b_1+b_2- \dots$ in terms of Betti numbers.
The minima of the functional $G \to \chi(G)$ however appear
difficult to compute \cite{eveneuler}. From the expectation formula
${\rm E}_{n,p}[\chi] = \sum_{k=1}^n (-1)^{k+1} \B{n}{k} p^{\B{k}{2}}$
\cite{randomgraph} of $\chi$ on Erd\"os-Renyi spaces we know that unexpectedly
large or small values of $\chi(G)$ can occur, even so we can not construct them directly.
As the expectation of $\chi = b_0-b_1+b_2 - \dots$ grows exponentially with the number of 
vertices. Also the total sum of Betti numbers grows therefore exponentially even so the
probabilistic argument gives no construction. We have no idea to construct a complex
with 10000 simplices for which the total Betti number is larger than say $10^{100}$ even so
we know that it exists as there exists a complex $G$ for which $\chi(G)$ is larger
than $10^{100}$. Such a complex must be a messy very high-dimensional Swiss cheese. 

\paragraph{}
After having done some experiments, we first felt that $\eta(G)$ must be non-negative. 
But this is false. In order to have negative Euler
characteristic for a unit sphere of a two-dimensional complex, we need already to have some
vertex for which $S(x)$ is a bouquet of spheres. 
A small example with $\eta(G)<0$ is obtained by taking a sphere, then glue in a disc into the
inside which is bound by the equator. This produces a geometry 
$G$ with Betti vector $(1,0,2)$ and Euler characteristic $3$. It satisfies $\eta(G)=-8$
as every of the $8$ vertices at the equator of the Barycentric refinement of $G_1$ has
curvature $\chi(S(x))=-1$ and for all the other vertices have $\chi(S(x))=0$. 

\begin{figure}[!htpb]
\scalebox{0.1}{\includegraphics{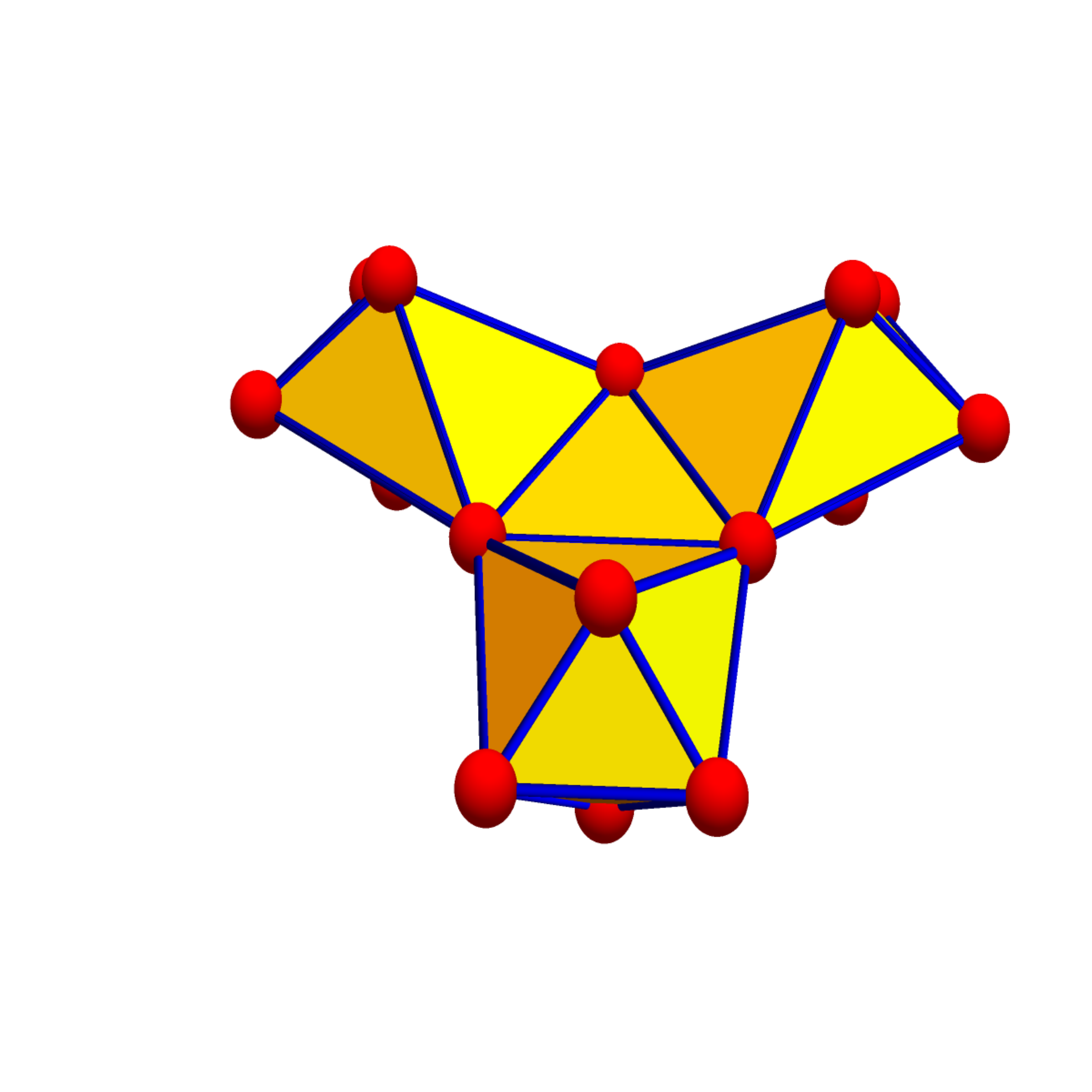}}
\scalebox{0.1}{\includegraphics{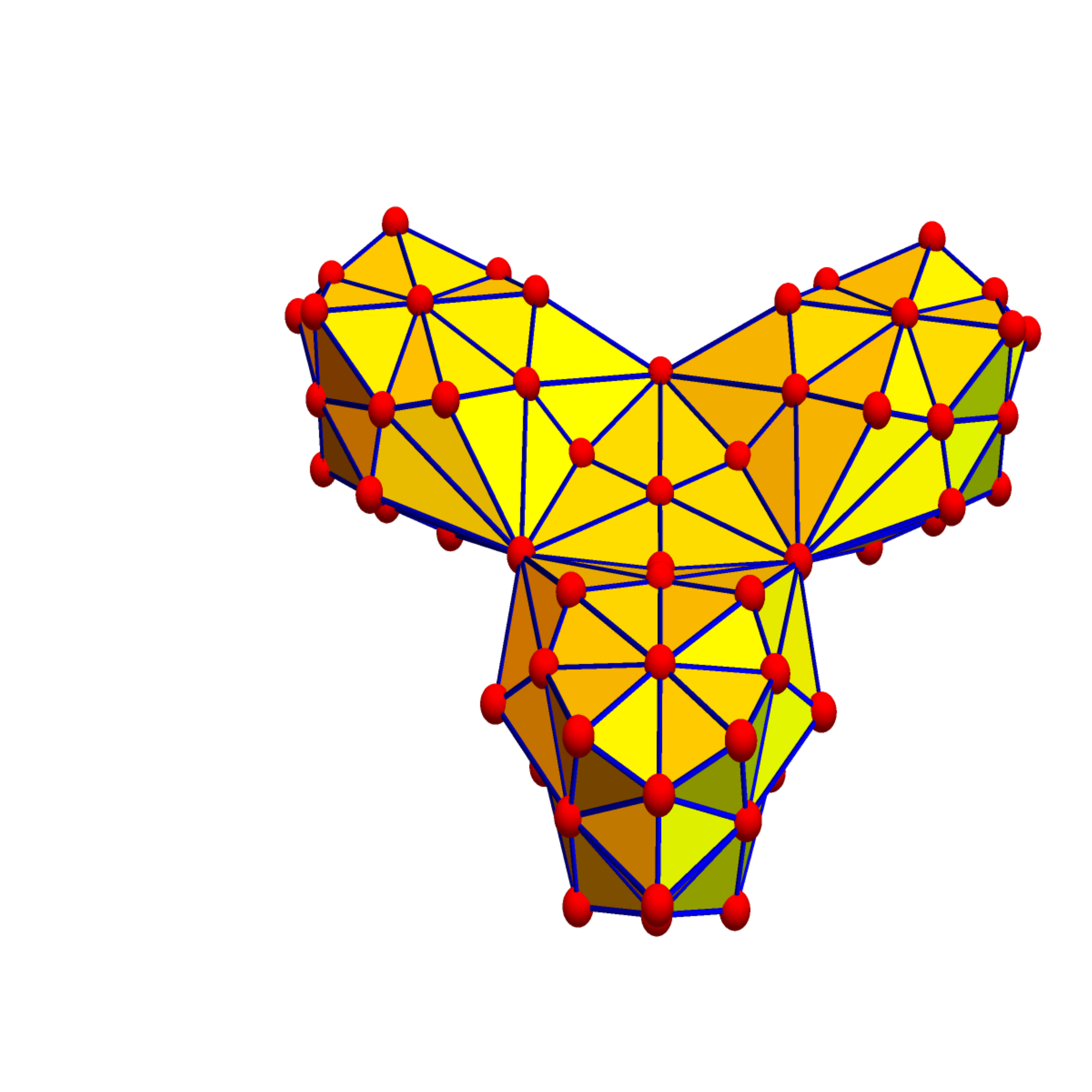}}
\caption{
A $2$-dimensional complex with $\eta(G)=-6<0$. The $f$-vector is
$(v,e,f)=(15,36,25)$, the Betti numbers are $b_0=1,b_1=0,b_2=3$, the
Euler characteristic $v-e+f=b_0-b_1+b_2=4$.
In this case, $2e-3f=(15,36,25) \cdot (0,2,-3)=-3$.
To the right, we see the refinement with $f$-vector
$(76,222,150)$ and $\eta(G_1)=2 \cdot 222 - 3 \cdot 150=-6$.
In general, the value doubles under refinement. 
}
\end{figure}

\paragraph{}
This example shows that $\eta(G)$ can become arbitrarily 
small even for two-dimensional complexes. But what happens in this example
there is a one dimensional singular set. It is the circle along which the 
disk has been glued between three spheres. We have not yet found an example
of a complex $G$ for which $G_1$ has a discrete set of singularities (vertices
where the unit sphere is not a sphere.) In the special case where $G$
is the union a finite set of geometric graphs with boundary in such a way that
the intersection set is a discrete set, then $\eta(G) \geq 0$. 

\section{Old results}

\paragraph{}
Given a face $x$ in $G$, it is also a vertex in $G_1$. The {\bf dimension} ${\rm dim}(x) = |x|-1$
with cardinality $|x|$ now defines a function on the vertex set of $G_1$. It is locally injective and 
so a {\bf coloring}. We know already $g(x,x) = 1-\chi(S(x))$ \cite{Spheregeometry} and that 
$V(x) = \sum_y g(x,y) = (-1)^{{\rm dim}(x)} g(x,x)$
is {\bf curvature}: $\sum_x V(x)=\chi(G)$. It is dual to the curvature $\omega(x) = (-1)^{{\rm dim}(x)}$ 
for which Gauss-Bonnet $\sum_x \omega(x)$ is the definition of Euler characteristic. Both of these
formulas are just Poincar\'e-Hopf for the gradient field defined by the function  ${\rm dim}$. The
Gauss-Bonnet formula $\sum_x V(x)= \chi(G)$ can be rewritten as $\sum_{x,y} g(x,y) = \chi(G)$. 
We call this the {\bf energy theorem}. It tells that the total potential energy of a simplicial 
complex is the Euler characteristic of $G$. By the way, $\sum_{x,y} L(x,y) = |V(G')|+2 |E(G')|$
by Euler handshake. 

\paragraph{}
If $v_k$ counts the number of $k$-dimensional faces of $G$, then $f(x)=\sum_{k=1}^{\infty} v_{k-1} x^k = 
v_0 x + v_1 x^2 + \dots $ is a {\bf generating function} for the $f$-vector $(v_0,v_1, \dots)$ of $G$. 
We can rewrite the Euler characteristic of $G$ as $\chi(G)=-f(-1)=f(0)-f(-1)$. If $G$ is a graph, we
assume it to be equipped with the {\bf Whitney complex}, the finite abstract simplicial complex 
consisting of the vertex sets of the complete subgraphs of $G$. This in particular applies for the
graph $G_1$. The $f$-vector of $G_1$ is obtained from the $f$-vector of $G$ by applying the matrix
$S_{ij} = i! S(j,i)$, where $S(j,i)$ are {\bf Stirling numbers} of the second kind. Since the 
transpose $S^T$ has the eigenvector 
$(1,-1,1,-1, \dots)$, the Euler characteristic is invariant under the process of taking
Barycentric refinement. Actually, as $S$ has simple spectrum, it is up to a constant the unique 
valuation of this kind. Quantities which do not change under Barycentric refinements are called 
{\bf combinatorial invariants}. 

\paragraph{}
The matrices $A,L,g$ act on a finite dimensional Hilbert space whose dimension is the 
number of faces in $G$ which is the number of vertices of $G_1$ or $G'$. Beside the usual 
trace ${\rm tr}$ there is now a {\bf super trace} ${\rm str}$ defined as 
${\rm str}(L)= \sum_x \omega(x) L(x,x)$ with $\omega(x) = (-1)^{{\rm dim}(x)}$. 
The definition $\chi(G) = \sum_x \omega(x)$ can now be written as ${\rm str}(1)= \chi(G)$. 
Since $L$ has $1$'s in the diagonal, we also have ${\rm str}(L)= \chi(G)$. A bit less obvious
is $\chi(G) = {\rm str}(g)$ which follows from the Gauss-Bonnet analysis leading to the
energy theorem. It follows that the Hydrogen operator $H$ satisfies ${\rm str}(H)=0$,
the super trace of $H$ is zero. This leads naturally to the question about the trace of $H$. 
By the way, the super trace of the Hodge Laplacian $L=(d+d^*)^2$ where $d$ is the exterior
derivative is always zero by Mc-Kean Singer (see \cite{knillmckeansinger} for the discrete case). 

\paragraph{}
The Barycentric refinement graph $G_1$ and the connection graph $G'$ have appeared also
in a number theoretical setup. If $G$ is the countable complex consisting of all finite subsets
of prime numbers, then the finite prime graph $G_1(n) \subset G_1$ has as vertices all square-free integers in
$V(n)=\{2,3,4 \dots ,n\}$, connecting two if one divides the other. The prime connection graph $G'(n)$
has the same vertices than $G_1(n)$ but connects two integers if they have a common factor larger than $1$.
This picture interprets sets of integers as simplicial complexes and sees
counting as a Morse theoretical process \cite{CountingAndCohomology}.
Indeed $\chi(G_1(n)) = 1-M(n)$, where $M(n)$ is the Mertens
function. If the vertex $n$ has been added, then $i(n)=1-\chi(S(n))=-\mu(n)$ with M\"obius function
$\mu$ is a Poincar\'e-Hopf index and $\sum_x i(x)=\chi(G_1(x))$ is a Poincar\'e-Hopf formula. In
combinatorics, $-i(G)=\chi(G)-1$ is called the {\bf reduced Euler characteristic} \cite{Stanley86}.
The counting function $f(x)=x$ is now a discrete Morse function, and each vertex is a
critical point. When attaching a new vertex $x$, a handle of dimension $m(x)={\rm dim}(S(x))+1$ is added.
Like for critical points of Morse functions
in topology, the index takes values in $\{-1,1\}$ and $i(x)=(-1)^{m(x)}$. For the
connection Laplacian adding a vertex has the effect that the determinant gets multiplied by $i(x)$.
Indeed, ${\rm det}(L)=\prod_x \omega(x)$ in general while $\chi(G) = \sum_x \omega(x)$,
if $\omega(x) = (-1)^{{\rm dim}(x)}$.

\section{The functional}

\paragraph{}
Define the functional 
$$   \eta(G) = {\rm tr}(H) = {\rm tr}(L - L^{-1})\; . $$ 
Due to lack of a better name, we call it the {\bf hydrogen trace}. We can rewrite this 
functional in various ways. For example $\eta(G) = \sum_k \lambda_k -1/\lambda_k$, 
where $\lambda_k$ are the eigenvalues of $L$. 
We can also write $\eta(G) = \sum_k \mu_k \frac{2+\mu_k}{1+\mu_k}$ where $\mu_k$ are the
eigenvalues of the adjacency matrix $A$ of the connection graph $G'$. 
It becomes interesting however as we will be able to link $\eta$ explicitly 
with the $f$-vector of the complex $G_1$ or even with the $f$-vector of the complex $G$ itself.

\paragraph{}
We will see below that also the {\bf Green trace functional} ${\rm tr}(g)$ is interesting
as $g=L^{-1}$ is the Green function of the complex. It is bit curious that there are
analogies and similarities between the Hodge Laplacian $H  = (d+d^*)^2$ of a complex and the 
connection Laplacian $L$. Both matrices have the same size. As they work on a space of
simplices, where the dimension functional defines a parity, one can also look at the
{\bf super trace} 
${\rm str}(L) = \sum_{{\rm dim}(x)>0} L_{xx} - \sum_{{\rm dim}(x)<0} L_{xx}$
It is a consequence of Mc-Kean Singer super symmetry that 
${\rm str}(H) = 0$
which compares with the definition ${\rm str}(L)=\chi(G)$ and leads to the McKean
Singer relation ${\rm str}(e^{-tH})=\chi(G)$. We have seen however the Gauss-Bonnet relation
${\rm str}(g)  = \chi(G)$
which implies the energy theorem $\sum_{x,y} g_{xy} = \chi(G)$.  It also implies 
${\rm str}(L-g) = 0$. 

\paragraph{}
The invertibility of the connection Laplacian is interesting and lead to topological
relations complementing the topological relations of the Hodge Laplacian to topology
like the Hodge theorem telling that the kernel of the $k$'th block of $H$ is isomorphic
the $k$'th cohomology of $G$. Both $L$ and $H$ have deficits: 
we can not read off cohomology of $L$ but we can not invert $H$, the reason for the
later is exactly cohomology as harmonic forms are in the kernel of $H$. 
So, there are some complementary benefits of both $L$ and $H$. And then there
are similarities like ${\rm str}(H) = {\rm str}(L-L^{-1})=0$ and 
${\rm str}(e^{-H}) = {\rm str}(L^{-1}) = \chi(G)$. 

\section{Gauss-Bonnet}

\paragraph{}
The following {\bf Gauss-Bonnet} theorem for $\eta$ shows that its curvature at a face 
$x$ is the Euler characteristic of the unit sphere $S(x)$ in the Barycentric refinement 
$G_1$. We use the notation $\eta_0(G) = \sum_{x \in V(G)} \chi(S(x))$, $\eta_1(G)=\eta_0(G_1)$
and $\eta(G)={\rm tr}(L-L^{-1})$. 

\begin{thm}
Let $G$ be an arbitrary abstract finite simplicial complex. Then 
$$  {\rm tr}(L-L^{-1}) = \eta(G) = \eta_1(G) = \eta_0(G_1) = \sum_{x \in V(G_1)} \chi(S(x)) \; . $$
\end{thm}
\begin{proof} 
The diagonal elements of $g=L^{-1}$ has entries $(1-\chi(S(x))$. 
We therefore have have ${\rm tr}(g) = \sum_x (1-\chi(S(x))$. We also
have ${\rm tr}(G) = \sum_x 1$. 
\end{proof} 

{\bf Examples.} \\
{\bf 1)} If $G=C_n$, then $G_1=C_{2n}$. Now, $\chi(S(x))=2$ for all vertices $x \in G_1$. 
We see that $\eta(C_n) = 4n$.  \\ 
{\bf 2)} For a discrete two-dimensional graph $G$, a graph for which every unit sphere is a circular
graph, we have $\eta(G) = 0$. \\
{\bf 3)} For a discrete three-dimensional graph $G$, a graph for which every unit sphere is a two 
dimensional sphere, we have $\eta(G) = 2 V(G_1) = 2 \sum_{k=0}^{\infty} v_k(G)$.   For
example, for the 3-sphere, the suspension of the octahedron, which can be written as
$G=3 P_2 = P_2+P_2+P_2$, we have $\eta(G) = 160$ because the $f$-vector of $G$ is 
$\vec{v} = (8, 24, 32, 16)$. \\
{\bf 4)} For a graph without triangles, we have $\eta(G) = \sum_{x \in V(G_1)} {\rm deg}(x)$
which is by handshaking $2 v_1(G_1)$. Since Barycentric refinement doubles the edges, 
we have $\eta(G)=4 v_1(G)$. This generalizes the circular case discussed above. \\
{\bf 5)} For $G=K_n$ we have $\eta(G) = 2^{n}$ if $n$ is even and $2^n-2$ if $n$ is odd. The 
numbers start as following: $\eta(K_1)=0$ $\eta(K_2)=4$, 
$\eta(K_3)=6$, $\eta(K_4)=16$, $\eta(K_5)=30$ etc. 

\begin{figure}[!htpb]
\scalebox{0.1}{\includegraphics{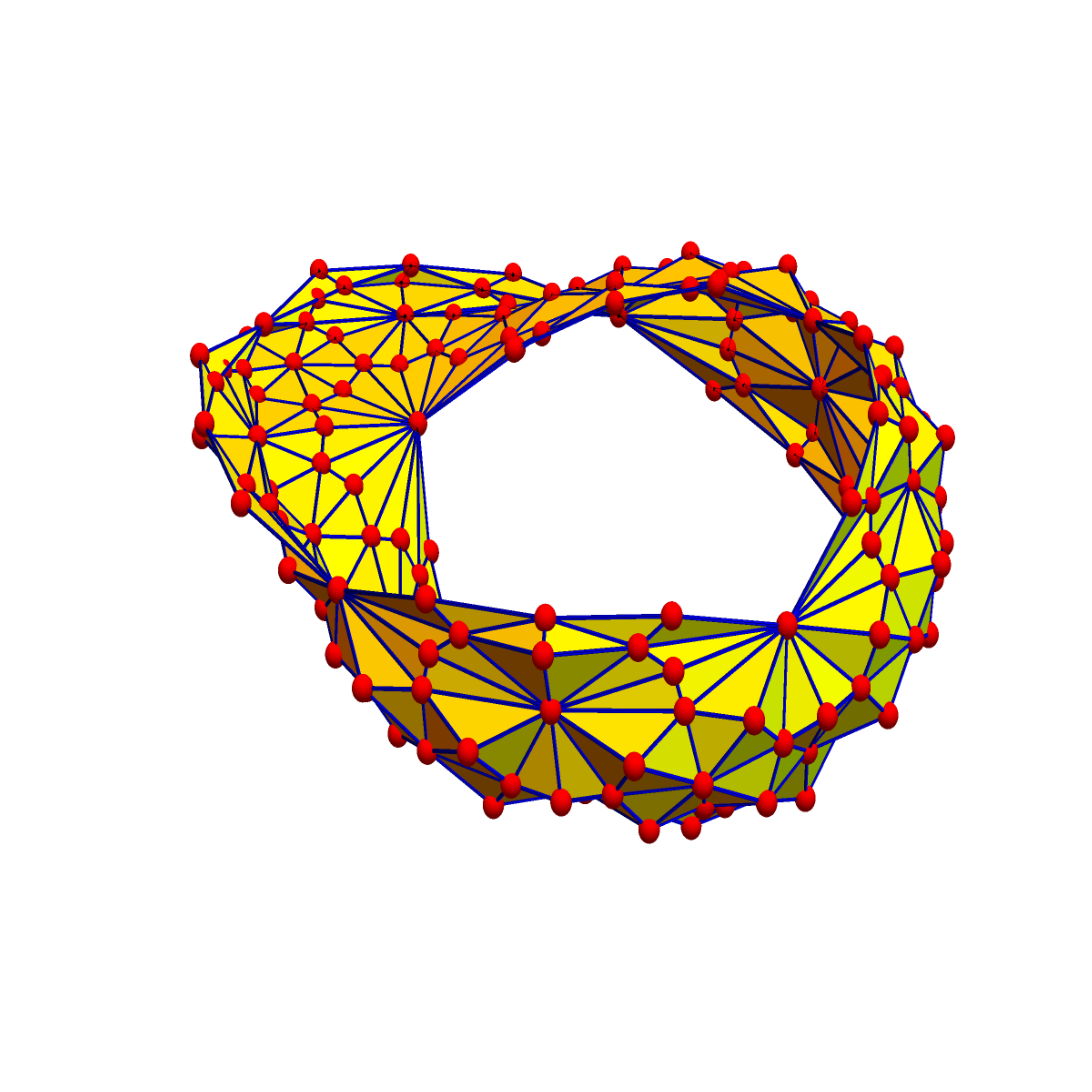}}
\caption{
The M\"obius strip is a $2$-dimensional graph with connected boundary. 
As $\chi(S(x))$ is zero in the interior and $\chi(S(x))=1$ at the boundary 
we see that $\eta(G)$ is the length of the boundary. 
In the displayed example of the discrete Moebius strip, we have $\eta(G)=32$.
}
\end{figure}

\section{Generating function}

\paragraph{}
Let $f_{G_1}(x) = 1+v_1 x + v_2 x^2 + \dots = 1+\sum_{k=1}^{\infty} v_{k-1} x^k$ be the 
(reduced) generating function for the Barycentric refinement $G_1$ of $G$. The {\bf Zykov join}
of two graphs $G_1 + H_1$ is defined as the graph with vertex set $V(G_1) \cup V(H_1)$
for which two vertices $a,b$ are connected if they were connected in $G_1$ or $H_1$ or if
$a,b$ belong to different graphs. The generating function of the sum $G_1+H_1$
is the product of the generating functions of $G_1$ and $H_1$.  \\

Since the Euler characteristic satisfies 
$\chi(G) = f(0)-f(-1) = \chi(G_1) = f_1(0)-f_1(-1)$, the following again
shows that the functional $\eta$ appears natural; 

\begin{coro}
$\eta(G) = f_1'(0)-f_1'(-1)$. 
\end{coro}

\paragraph{}
To prove this, we rewrite the Gauss-Bonnet result as a Gauss-Bonnet result for the 
second Barycentric refinement $G_2$. Define for a vertex $x$ in $G_2$ the curvature 
$$   k(x) = (-1)^{1+{\rm dim}(x)} (1+{\rm dim}(x)) \; . $$

\begin{lemma}
$\eta(G) = \sum_{x \in V(G_2)} k(x)$, where the sum is over all vertices $x$ in $G_2$ which have
positive dimension.
\end{lemma}
\begin{proof}
This is a handshake type argument. We start with $\eta(G) = \sum_x \chi(S(x))$. 
Since every $d$-dimensional simplex in $S(x)$ defines a $(d+1)$-dimensional simplex
containing $x$, super summing over all simplices of $S(x)$ gives a super sum over
simplices in $G_2$ where each simplex such $y$ appears ${\rm dim}(y)+1$ times. 
\end{proof}

{\rm Remark.} This gives us an upper bound on the functional $\eta$ in terms of the 
number of vertices in $G_2$ and the maximal dimension of $G$: $\eta(G) \leq |V(G_2)| (1+d)$.
If $G_1$ has $n$ elements, then $G_2$ has $\leq (d+1)! n$ elements. We see:

\begin{coro}
$\eta(G)$ is bounded above by $C_d |V(G_1)|$, where $C_d$ only depends on the maximal dimension of $G$.
\end{coro}

\paragraph{}
Now we can prove the result:

\begin{proof}
As $f_1(x) = v_0 x + v_1 x^2 + v_2 x^3 + \dots$, we have 
$f_1'(x) = v_0 + 2 v_1 x + 3 v_2 x^2 + \dots$ and
$f_1'(0)-h_1'(-1) = 2 v_1 - 3 v_2 + 4 v_3$ which is the same than 
$\sum_{x, {\rm dim}(x)>0} (-1)^{1+{\rm dim}(x)} (1+{\rm dim}(x))$. 
\end{proof}

As an application we can get a formula for $\eta_0(G_1+H_1)$, where $G_1+H_1$
is the Zykov sum of $G_1$ and $H_1$. The Zykov sum shares the properties
of the classical join operation in the continuum. The Grothendieck argument
produces from the monoid a group which can be augmented to become a ring
\cite{Spheregeometry}.

\begin{coro}
On the set of complexes with zero Euler characteristic, we have
$\eta_0(G_1+H_1) = \eta_0(G_1) + \eta_0(H_1)$.
\end{coro}
\begin{proof}
We have $f_{G_1 + H_1} = f_{G_1} f_{H_1}$.
Now $\eta_0(G_1) = f_{G_1}'(0)-f_{G_1}'(-1)$ and
$\eta_0(H_1) = f_{H_1}(0)-f_{H_1}'(-1)$. By the product rule,
$f_{G_1+H_1}' = (f_{G_1} f_{H_1})' = f_{G_1}' f_{H_1} + f_{G_1} f_{H_1}'$.
we have now $f_{G_1+H_1}'(0) = f_{G_1}'(0) + f_{H_1}'(0)$ and
$f_{G_1+H_1}'(-1) =  f_{G_1}'(-1) (1-\chi(H_1)) + (1-\chi(G_1)) f_{H_1}'(-1)$ so that
$\eta_0(G_1+H_1) = \eta_(G_1) + \eta_(G_2) + f_{G_1}'(-1) \chi(H_1) + f_{H_1}'(-1) \chi(G_1)$. 
\end{proof} 

\section{Geometric graphs} 

\paragraph{}
We will see in this section that for graphs which discretized manifolds
or varieties which have all singularities isolated and split into such discrete manifolds,
the functional $\eta$ is non-negative. A typical example is a bouquet of spheres, 
glued together at a point.

\begin{figure}[!htpb]
\scalebox{0.1}{\includegraphics{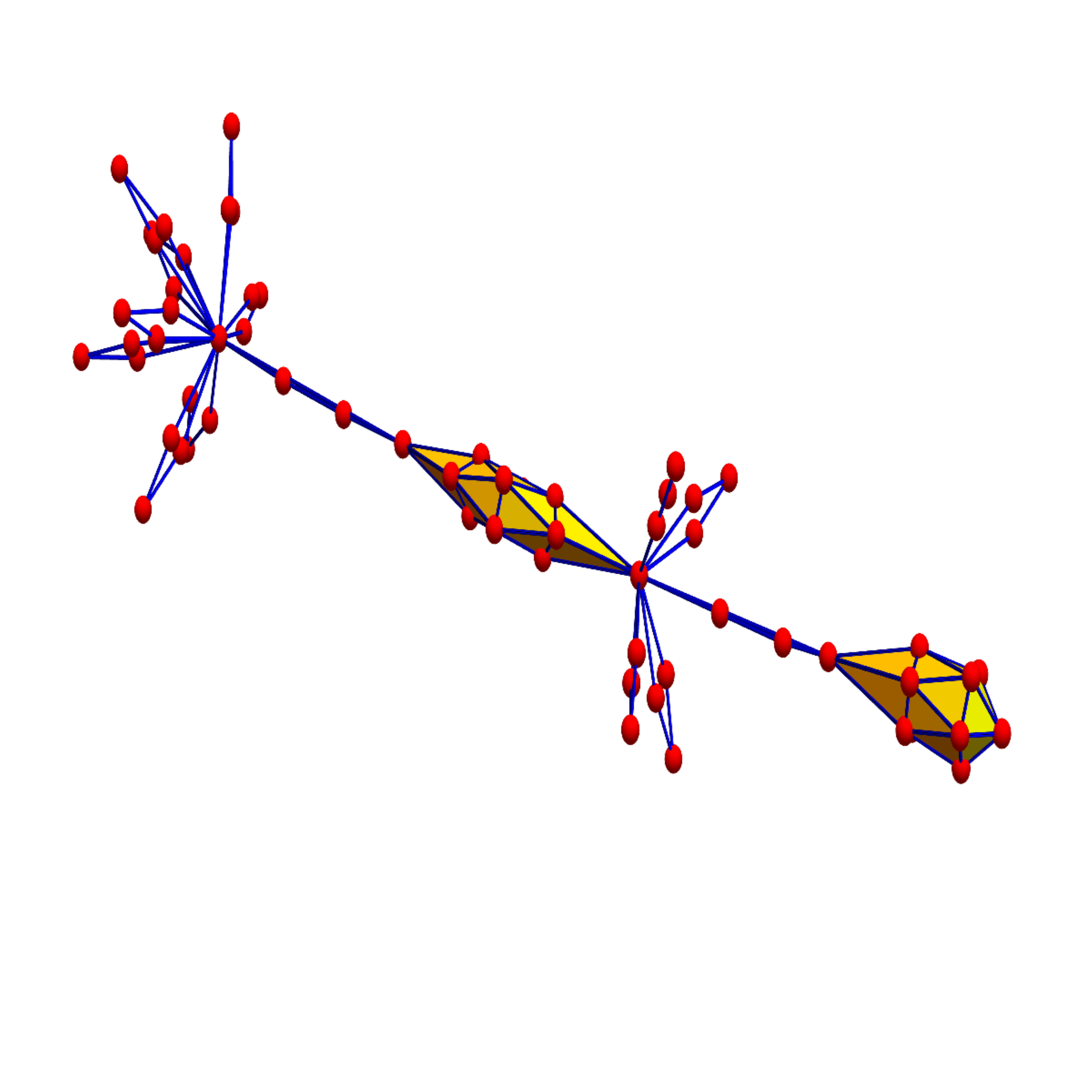}}
\caption{
An example of a discrete variety. A graph for which all unit spheres are
discrete spheres, with some exceptional but isolated points, the singularities. 
}
\end{figure}

\paragraph{}
A {\bf $d$-graph} is a finite simple graph for which every unit sphere $S(x)$ is a $(d-1)$-graph
which is a $(d-1)$-sphere. A {\bf $d$-sphere} is a $d$-graph which becomes collapsible if a
single vertex is removed. The inductive definitions of $d$-graph and $d$-sphere start with the
assumption that the empty graph is a $-1$-sphere and $-1$ graph and that the $1$ point graph
$K_1$ is collapsible. A graph $G$ is {\bf collapsible} if there exists a vertex $x$ such that both
$G \setminus x$ and $S(x)$ are collapsible. It follows by induction that $d$-sphere has Euler
characteristic $1+(-1)^{d} \in \{ 0,2\}$.

\paragraph{}
A simplicial complex is called a $d$-complex if its refinement $G_1$ is a $d$-graph. We now
see that for even-dimensional $d$-complexes, the functional $\eta$ is zero. A graph is a
{\bf $d$-graph with boundary} if every unit sphere is either a sphere or contractible and
such that the {\bf boundary}, the set of vertices for which the unit sphere $S(x)$ is contractible 
is a $d-1$-graph. An example is the wheel graph $G$ for which the boundary $\delta G$ 
is a circular graph.

\paragraph{}
Since for an even dimensional $d$-graph with boundary the Euler characteristic of the unit spheres
in the interior is zero, Gauss-Bonnet implies $\eta(G) = |V(\delta G)|$.
In the case of an odd-dimensional $d$-graph with boundary, it leads to
$\eta(G)=2 |{\rm int}(G)| + |\delta G|$. This leads to the observation: 

\begin{lemma}
If $G$ is a $d$-graph with boundary, then $\eta(G) \geq 0$. Equality holds
if and only if $G$ is an even dimensional graph without boundary. 
\end{lemma}

\paragraph{}
This can be generalized: if $G$ is a union of finitely many $d_k$-graphs $G_k$
such that the set of vertices which belong to at least $2$ graphs is isolated
in the sense that the intersection of any two $G_k$ does not contain any edge, 
then  $\delta(G) \geq 0$. Equality holds if $G$ is a finite union of even dimensional
graphs without boundary touching at a discrete set of points. The reason is that the unit 
spheres are again either spheres or then finite union of spheres of various dimension. 
Since the Euler characteristic of of a sphere is non-negative and the Euler characteristic
of a disjoint sum is the sum of the Euler characteristics, the non-negativity of $\eta$
follows. We will ask below whether more general singularities are allowed and still
have $\eta(G) \geq 0$.  \\

\paragraph{}
Maybe in some physical context, one would be interested especially in 
the case $d=4$ and note that
among all $4$-dimensional simplicial complexes with boundary
the complexes without boundary minimize the functional $\eta$. 
In the even dimensional case, the curvature of $\eta$ is supported
on the boundary of $G$. If we think of the curvature as a kind of 
charge, this is natural in a potential theoretic setup. Indeed, 
one should think of $V_x(y) = g(x,y)$ as a potential \cite{Helmholtz}.
In the case of an odd dimensional complex, there is a constant 
curvature present all over the interior and an additional constant
curvature at the boundary. Again, also in the odd dimensional case, the
absence of a boundary minimizes the functional $\eta$. \\

\paragraph{}
We should in this context also mention the {\bf Wu characteristic} for 
which we proved in \cite{valuation} that for $d$ graphs with boundary, the 
formula $\omega(G) = \chi(G)-\chi(\delta G))$ holds. 
The Wu characteristic $\omega$ was defined as 
$\omega(G) = \sum_{x \cap y \neq \emptyset} \omega(x) \omega(y)$ 
with $\sigma(x)=(-1)^{{\rm \omega}(x)}$. The Wu characteristic fits into
the topic of connection calculus as 
$\omega(G) = {\rm tr}(L J)$, where $J$ is the {\bf checkerboard matrix} 
$J_{xy} = (-1)^{{\rm dim}(i)+{\rm dim}(j)} = \omega^T \cdot \omega$
so that $J/n$ is a projection matrix \cite{Helmholtz}. Actually, in 
the eyes of Max Born, one could see $\omega(G)/n = (\omega,L \omega )$ 
as the expectation of the {\bf state} $\Omega=\omega/\sqrt{n}$.

\section{The sum of the sphere Euler characteristic}

\paragraph{}
We look now a bit closer at the functional
$$ \eta_0(G) = \sum_{x \in V(G)} \chi(S(x)) \;  $$
on graphs. It appears to be positive or zero for most Erd\"os-Renyi graphs
but it can take arbitrary large or small values. We have seen that
$\eta(G) = \eta_0(G_1) = {\rm tr}(L-L^{-1})$.
But now, we look at graphs $G$ which are not necessarily the Barycentric
refinement of a complex. \\

{\bf Examples.} \\
{\bf 1)} For a complete graph $K_{n}$ we have $\eta_0(G)=n$. \\
{\bf 2)} For a complete bipartite graph $K_{n,m}$ we have $\eta_0(G) = 2 n m$ and $\eta(G)=4 n m$.  \\
{\bf 3)} For an even dimensional $d$-graph $G$, we have $\eta_0(G)=0$. \\
{\bf 4)} For an odd dimensional $d$-graph $G$ we have $\eta_0(G) = 2|V(G)|$.  \\
{\bf 5)} For the product $G$ of linear graph $L_m$ of length $m$ with a figure $8$ graph $E =C_k \wedge_x C_k$
      we have for $k \geq 4$ and $m \geq 1$ the formula $\eta_0(L_m \times (C_k \wedge_x C_k)) = 
      28+(k-4) 8 - (m-1) 4 = 8k-4m$.  \\
{\bf 6)} So far, in all examples we have seen if $G_1$ is the Barycentric refinement, we see
      $|\eta_0(G_1)| \geq 2 |\eta_0(G)|$.  \\
{\bf 7)} For the graph $G$ obtained by filling in an equator plane, we have $\eta_0(G)=-4$ and $\eta(G)=-8$. 

\begin{figure}[!htpb]
\scalebox{0.1}{\includegraphics{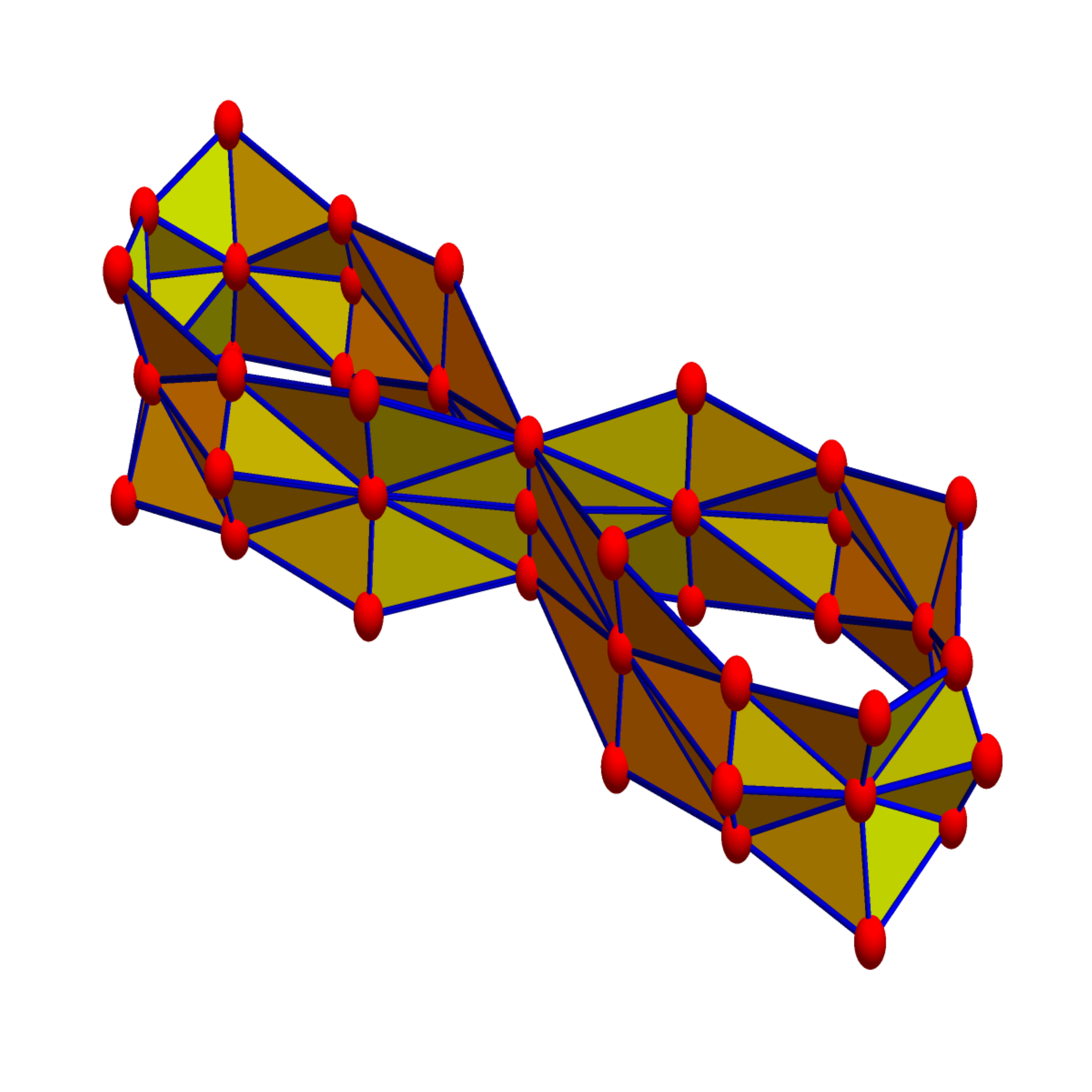}}
\caption{
The graph $G=L_2 \times (C_4 \wedge_x C_4)$ has $\eta_0(G)=28$. 
Increasing the length of $C_n$ by $1$ increases $\eta_0$ by $8$; increasing the
length of $L_m$ decreases $\eta_0$ by $4$. 
}
\end{figure}

\begin{lemma}
The functional is additive for a wedge sum and for the disjoint sum.
\end{lemma}
\begin{proof}
In both cases, the unit spheres $S(x)$ for vertices $x$ in one
of the graphs $H,G$ only are not affected. For the vertex in the intersection,
then $S_{G \cup H}(x)$ is the disjoint union $S_G(x) \cup S_H(x)$. 
\end{proof} 


\begin{coro}
For every homotopy type of graphs, the functional $\eta_0$ is both unbounded from above and
below.
\end{coro}
\begin{proof}
Take a graph $G$ with a given homotopy type. Take a second graph 
$H=\eta_0(L_m \times (C_k \wedge C_k))$ with large $m$. It has $\eta_0(H)=8k-4m$.
We can close one side of the graph to make it contractible. This produces a contractible
graph with $\eta_0(\tilde{H}) = 8k-2m$. The graph $\tilde{H} \wedge G$ now has the same
homotopy type than $G$ and has $\eta(\tilde{H} \wedge G) = 8k-2m + \eta_0(G)$. By choosing
$k$ and $m$ accordingly, we can make $\eta_0$ arbitrarily large or small. The addition 
of the contractible graph has not changed the homotopy type of $G$. 
\end{proof}

Let $C(k)$ denote the set of connected graphs with $k$ vertices.
On $C(2)$, we have $2 \leq \eta_0(G) \leq 2$,
on $C(3)$, we have $3 \leq \eta_0(G) \leq 3$,
on $C(4)$, we have $4 \leq \eta_0(G) \leq 8$,
on $C(5)$, we have $4 \leq \eta_0(G) \leq 12$,
and on $C(6)$ we have $0 \leq \eta_0(G) \leq 18$. 

\section{About the spectrum of $L$}

\paragraph{}
We have not found any positive definite connection Laplacian $L$ yet.
Since $L$ has non-negative entries, we know that $L$ 
has non-negative eigenvalues. By unimodularity \cite{Unimodularity} it therefore has 
some positive eigenvalue. The question about negative eigenvalues is 
open but the existence of some negative eigenvalues would follow from ${\rm tr}(H) \geq 0$ 
thanks to the following formula dealing with the column vectors $A_i=L_i-e_i$ of the 
adjacency matrix of $G'$.

\begin{lemma} 
$\eta(G) = -\sum_i (A_i, g A_i)$. 
\end{lemma}
\begin{proof}
Let $A$ be the adjacency matrix of the connection graph $G'$ so that the connection Laplacian $L$
satisfies $L=1+A$. As $L$ has entries $1$ in the diagonal only, we  
know ${\rm tr}(L) = n$ and $g \cdot A = (1+A)^{-1} A = -(1-(1+A)^{-1})= g-1$ so that
$\eta(G) = {\rm tr}(L)-{\rm tr}(g)={\rm tr}(1-g)=-{\rm tr}(g \cdot A) 
= - \sum_i e_i g A e_i = - \sum_i e_i g A_i) = - \sum_i A_i g A_i$.
The reason for the last step
is $\sum_i (e_i+A_i) g A_i = \sum_i L_i g A_i = \sum_i e_i A_i = 0$.
\end{proof}  

This immediately implies that $g$ (and so $L=g^{-1}$) can not 
be positive definite if $\eta(G) \geq 0$. Indeed, if $g$ were
positive definite, then $A_i g A_i>0$ for all $A_i$ and so $\sum_i A_i g A_i>0$
but $\eta(G) \geq 0$ implies $\sum_i A_i g A_i \leq 0$. 

\section{Open questions}

\paragraph{}
Since the entries of $L$ are non-negative and $L$ is invertible, the
Perron-Frobenius theorem shows that $L$ has some positive eigenvalue. 
A similar argument to show negative eigenvalues does not work yet, at 
least not in such a simple way. The next result could be easy but
we have no answer yet for the following question: 

\question{
Is it true that every connection Laplacian $L$ has some negative eigenvalue?
}


\paragraph{}
Lets call a graph a {\bf $d$-variety}, if every unit sphere is a $(d-1)$-graph
except for some isolated set of points, the singularities, where the unit sphere
is allowed to be a $(d-1)$-variety. In every example of a $d$-variety 
with $\eta(G)<0$ seen so far, we have the singularities non-isolated 

\question{
Is it true that $\eta(G) \geq 0$ if $G$ is a $d$-variety?
}

The inequality holds for $d$ graphs with boundary as for such 
graphs every unit sphere either has non-negative Euler characteristic 
$0,2$ (interior) or $1$ (at the boundary). The example shown above with $\eta(G)<0$
has some unit spheres which are not $1$-graphs (disjoint unions of circular graphs)
but $1$-varieties. \\

It follows that also for patched versions, graphs which are the union of two graphs such that
at the intersection, the spheres add up. This happens for example, 
if two disks touch at a vertex.  \\

\paragraph{}
If we think of $\chi(S(x))$ as a {\bf curvature} for the functional $\eta$, 
then a natural situation would be that zero total curvature implies that 
the curvature is zero everywhere. Here is a modification of the example
with negative $\eta$ for which $\eta(G)=0$ and so $\eta(G_n)=0$ for all $n$.

\begin{figure}[!htpb]
\scalebox{0.1}{\includegraphics{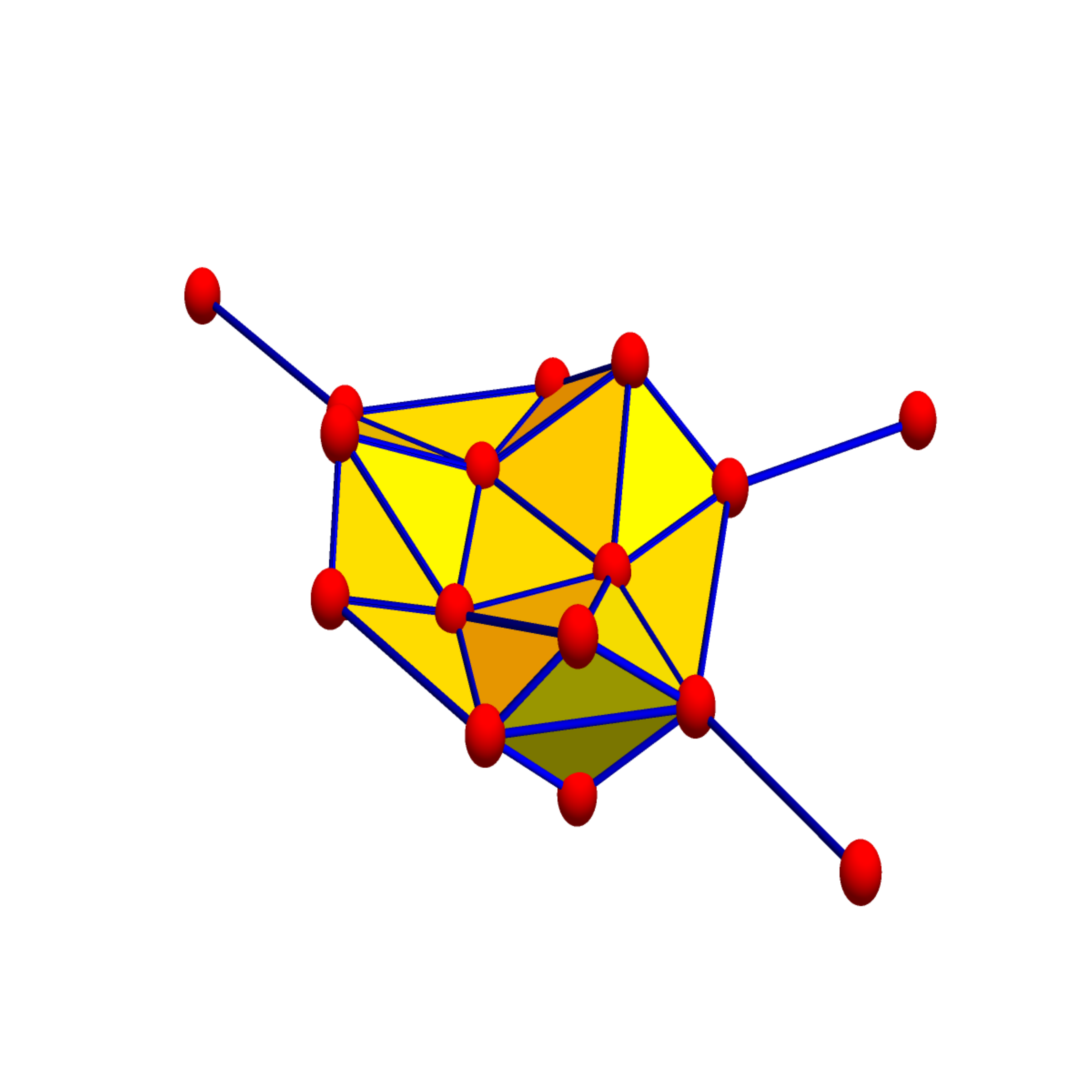}}
\scalebox{0.1}{\includegraphics{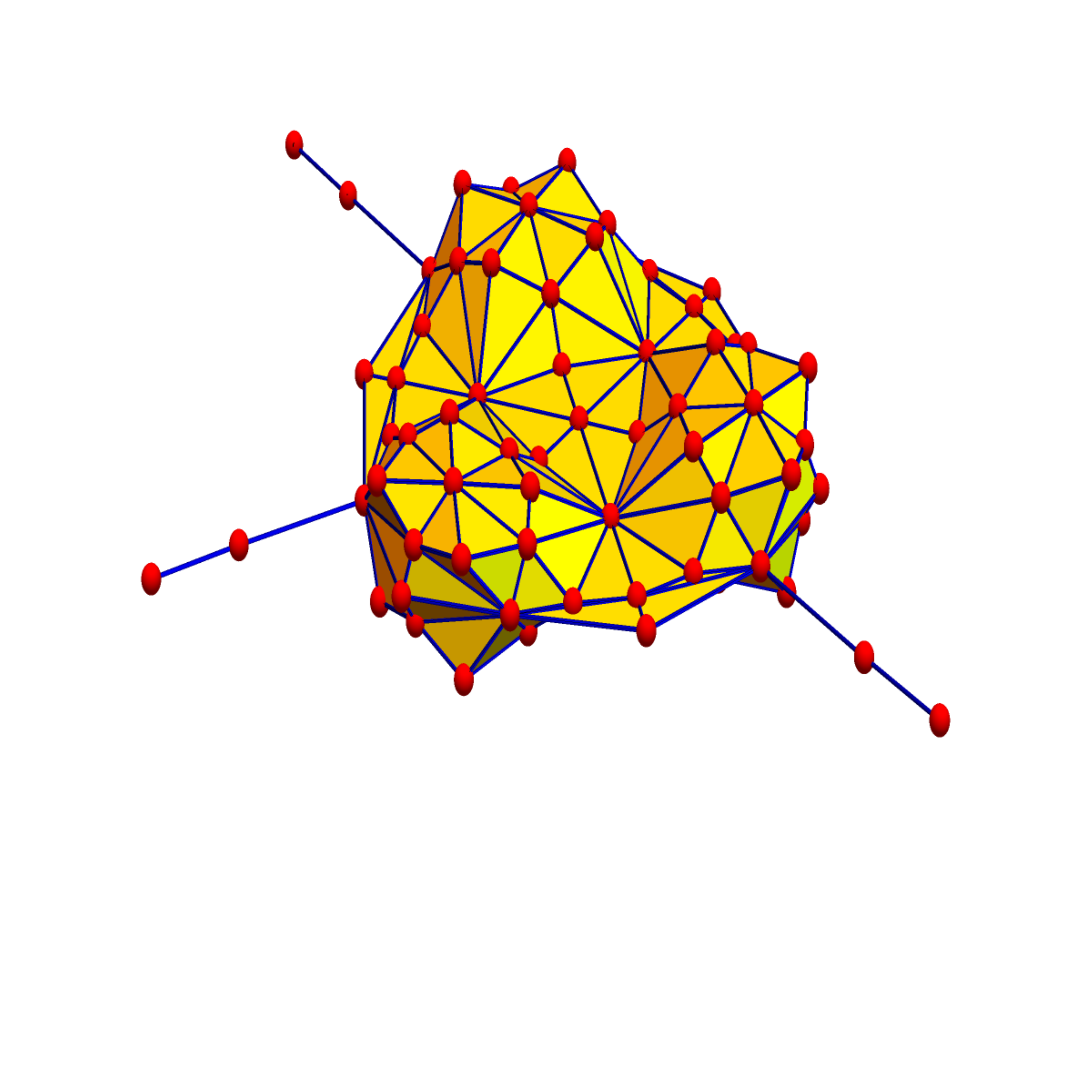}}
\caption{
A $2$-dimensional complex with $\eta(G)=0$. The $f$-vector is
$(v,e,f)=(18,42,28)$, the Betti numbers are still $b_0=1,b_1=0,b_2=3$, the
Euler characteristic still $v-e+f=b_0-b_1+b_2=4$. 
In this case, $2e-3f=(18,42,28) \cdot (0,2,-3)=0$. 
To the right, we see the refinement with $f$-vector 
$(88,252,168)$ and $\eta(G_1)=(88,252,168) \cdot (0,2,-3)$. Of course
$\eta(G_n) = S^n (18,42,28) (0,2,-3)=0$ for all $n$. 
}
\end{figure}

\paragraph{}
\question{
Is it true that for a $d$-variety, 
we have $\eta(G)=0$ if and only if $\chi(S(x))=0$ for all $x \in V(G_1)$. 
}

{\bf Examples:} \\
{\bf 1)} For $1$-dimensional graphs, we have $\eta(G)=4 v_1(G) > 0$. \\
{\bf 2)} For $2$-dimensional graphs, graphs with $K_3$ subgraphs and no $K_4$ 
subgraphs, the functional is a Dehn-Sommerville valuation $\eta(G) = 2 v_1 - 3 v_2$. It 
vanishes if every edge shares exactly two triangles. See \cite{valuation} for generalizations
to multi-linear valuations. \\
{\bf 3)} For $3$-dimensional graphs, graphs with $K_4$ subgraphs but no 
$K_5$ subgraphs, the functional is the Dehn-Sommerville valuation 
$2 v_1 - 3 v_2 + 4 v_3$.

\paragraph{}
A bit stronger but more risky is the question whether zero curvature implies 
that $G$ has the property that all unit spheres are unions of 
$d$-spheres:

\question{
Does $\eta(G)=0$ imply that $G$ is a finite union of disjoint $d_k$-graphs
wedged together so every unit sphere is a disjoint union of even dimensional $d$-graphs? }

\paragraph{}
Besides discrete manifolds, there are discrete varieties for which $\eta(G)=0$.
Here is an example: 

\begin{figure}[!htpb]
\scalebox{0.1}{\includegraphics{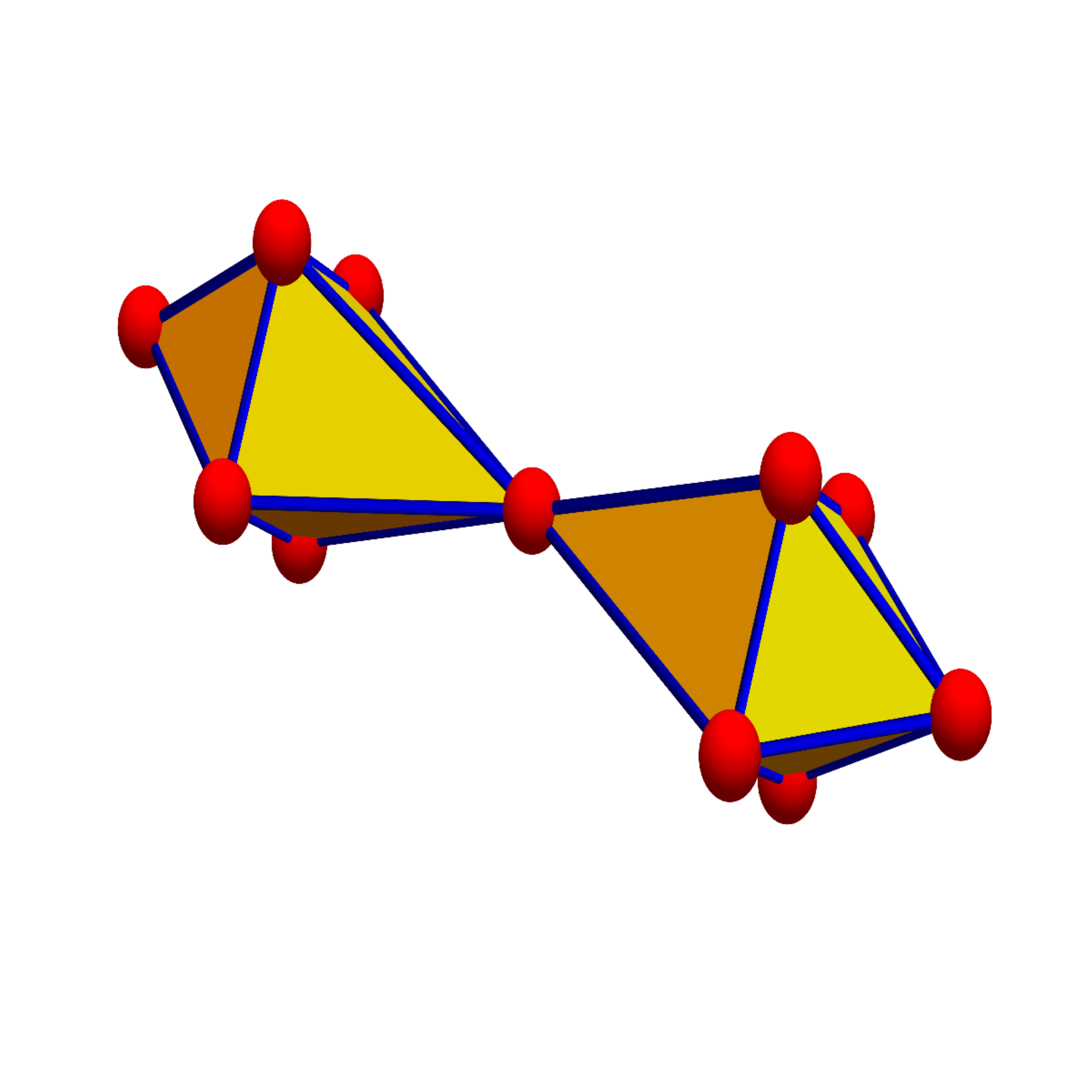}}
\scalebox{0.1}{\includegraphics{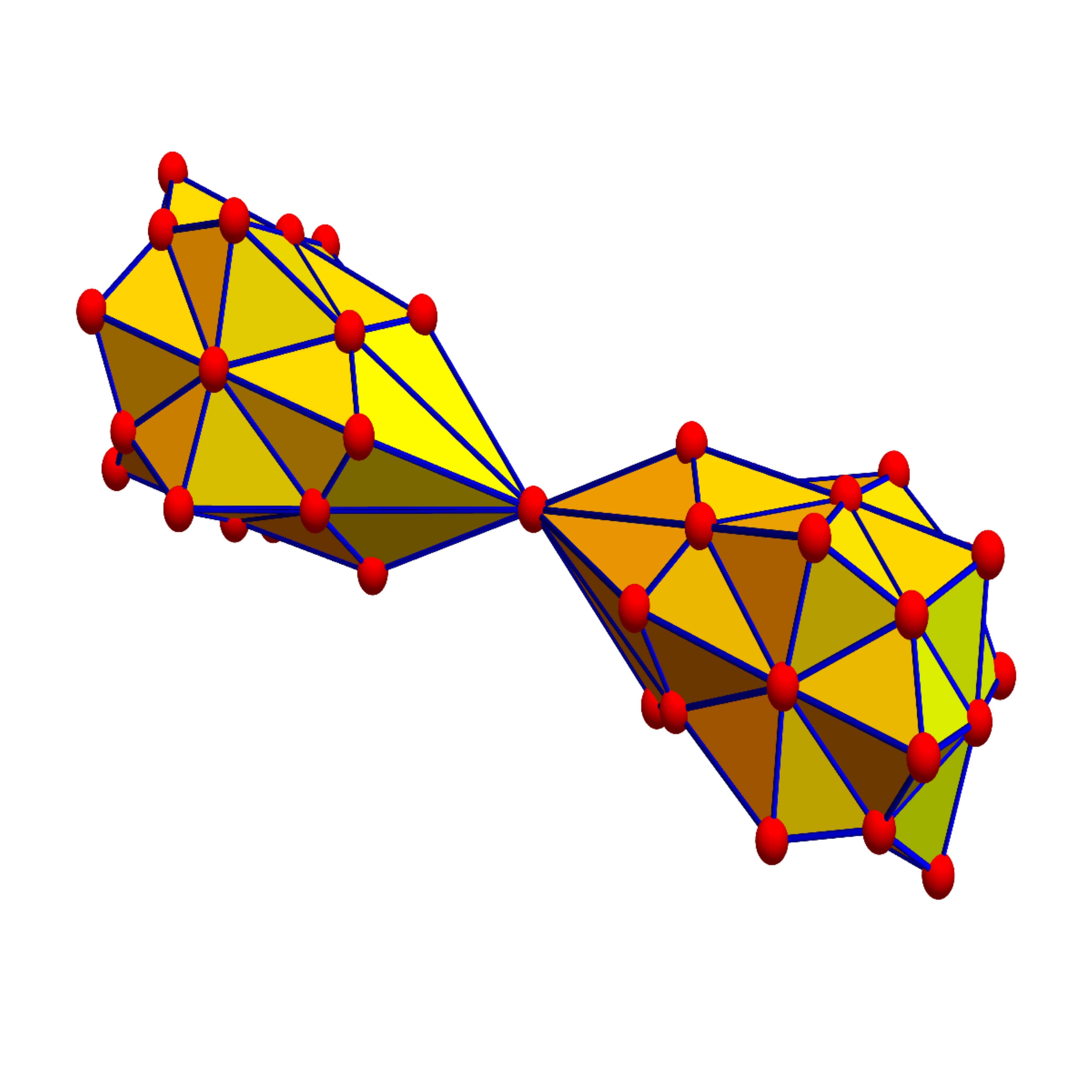}}
\caption{
A graph for which one of the unit spheres is a disconnected union of circular 
graphs. It is a $2$-variety but not a $2$-graph. 
It satisfies $\chi(G)=3$ and $\eta(G)=0$. To the right we see the Barycentric
refinement $G_1$ of $G$. The f-vector of $G_1$ is $\vec{v}=(51, 144, 96)$.
We have $\chi(G) = \vec{v} \cdot (1,-1,1)=3=b_0-b_1+b_2$ and 
$\eta(G) = \vec{v} \cdot (0,2,-3)=0$.  Again, for all Barycentric refinements,
we have $\eta(G)=0$. 
}
\end{figure}


\paragraph{}
We could imagine for example that there are graphs which are almost
$d$-graphs in the sense that the unit spheres can become discrete homology spheres, graphs
with the same cohomology groups as spheres but whose geometric realiazations are not 
homeomorphic to spheres. An other possibility is we get graphs for which the Euler characteristic
is zero but for which are also topologically different from spheres. We could imagine generalized
$4$-graphs for example, where some unit spheres are $3$-graphs. All odd-dimensional 
graphs have then zero Euler characteristic by Dehn-Sommerville (an incarnation of Poincar\'e
duality). But we don't know yet of a construct of such graphs. 

\paragraph{}
One can look at further related variational problems on graphs. One can either fix the number
of elements in $G$ or the number of elements in $G_1$. In the later case, it of course does not
matter whether one minimizes ${\rm tr}(L-g)$ or maximizes the trace of the Green function ${\rm tr}(g)$
as ${\rm tr}(L)=f(0)=|V(G_1)|=n$. 

\question{
What are the minima of $\eta$ on complexes with a fixed number $n$ of faces?
Equivalently, what are the maxima of the trace of the Green function on complexes
with $n$ faces? }

This is a formidable problem if one wants to explore it numerically  as the number of 
simplicial complexes with a fixed number $n$ of faces grows very fast. 
A good challenge could be $n=26$ already as the $f$-vector of the octahedron $G$ is $(6,12,8)$. The
simplex generating function of $G_1$ is $f(x) = 1+26x+72x^2+48 x^3$ with Euler's Gem $\chi(G)=f(0)-f(-1)=2$.
Since $f'(x)=26+144x+144 x^2$ we have $f'(-1) = 26 = {\rm tr}(g)$. 
It is a good guess to ask whether the trace of a Green function $g$ of a simplicial complex $G$ 
with $26$ simplices can get larger than $26$. In Figure~(\ref{octahedronmatrix}) we
look at the $26 \times 26$ matrices $L$ and $g=L^{-1}$ of the Octahedron complex. 
Both matrices have $1$ in the diagonal. 
In the Green function case, we know $g_{xx} = 1-\chi(S(x)) = 1$ as all unit spheres $S(x)$ in $G_1$ 
are circles

\begin{figure}[!htpb]
\scalebox{0.11}{\includegraphics{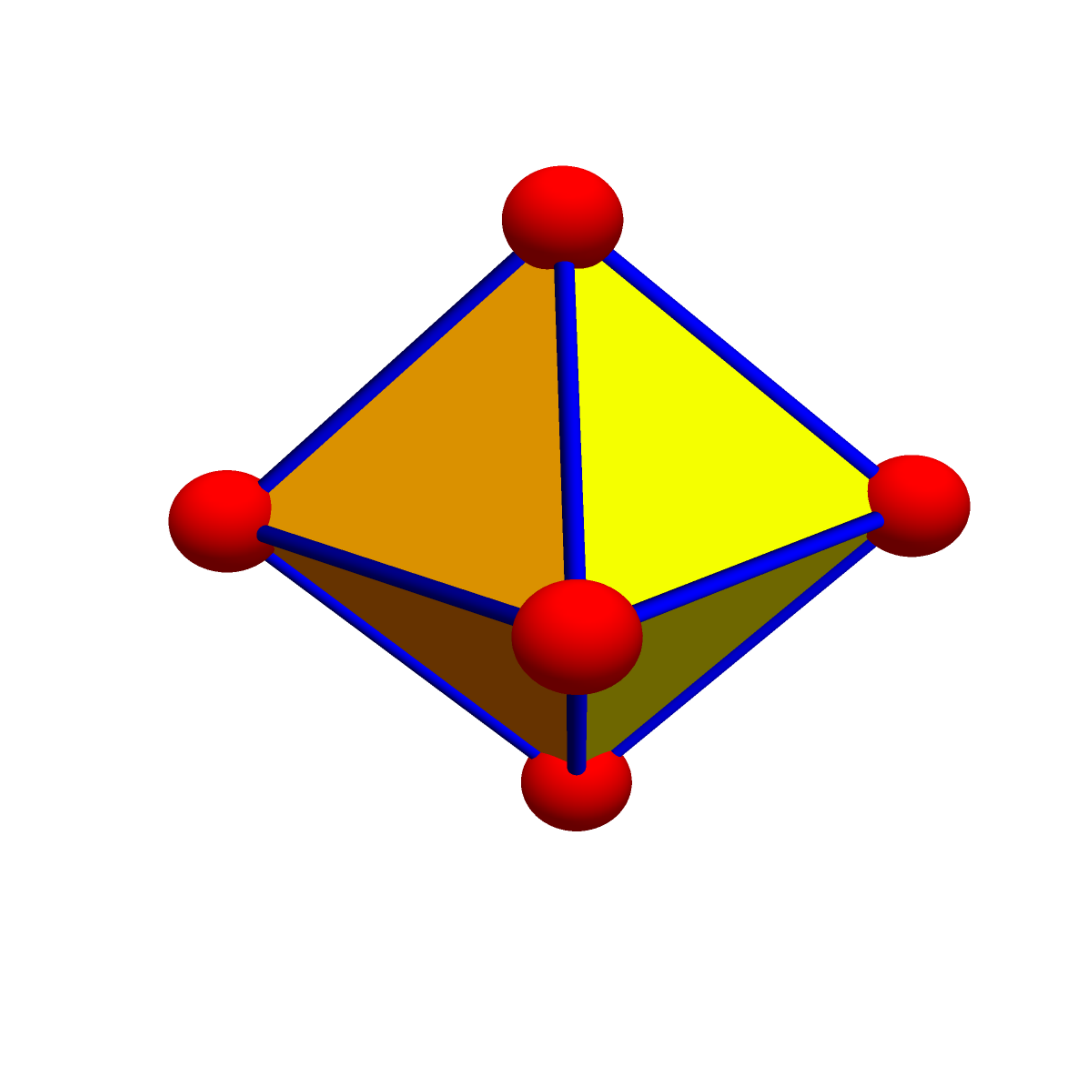}}
\scalebox{0.11}{\includegraphics{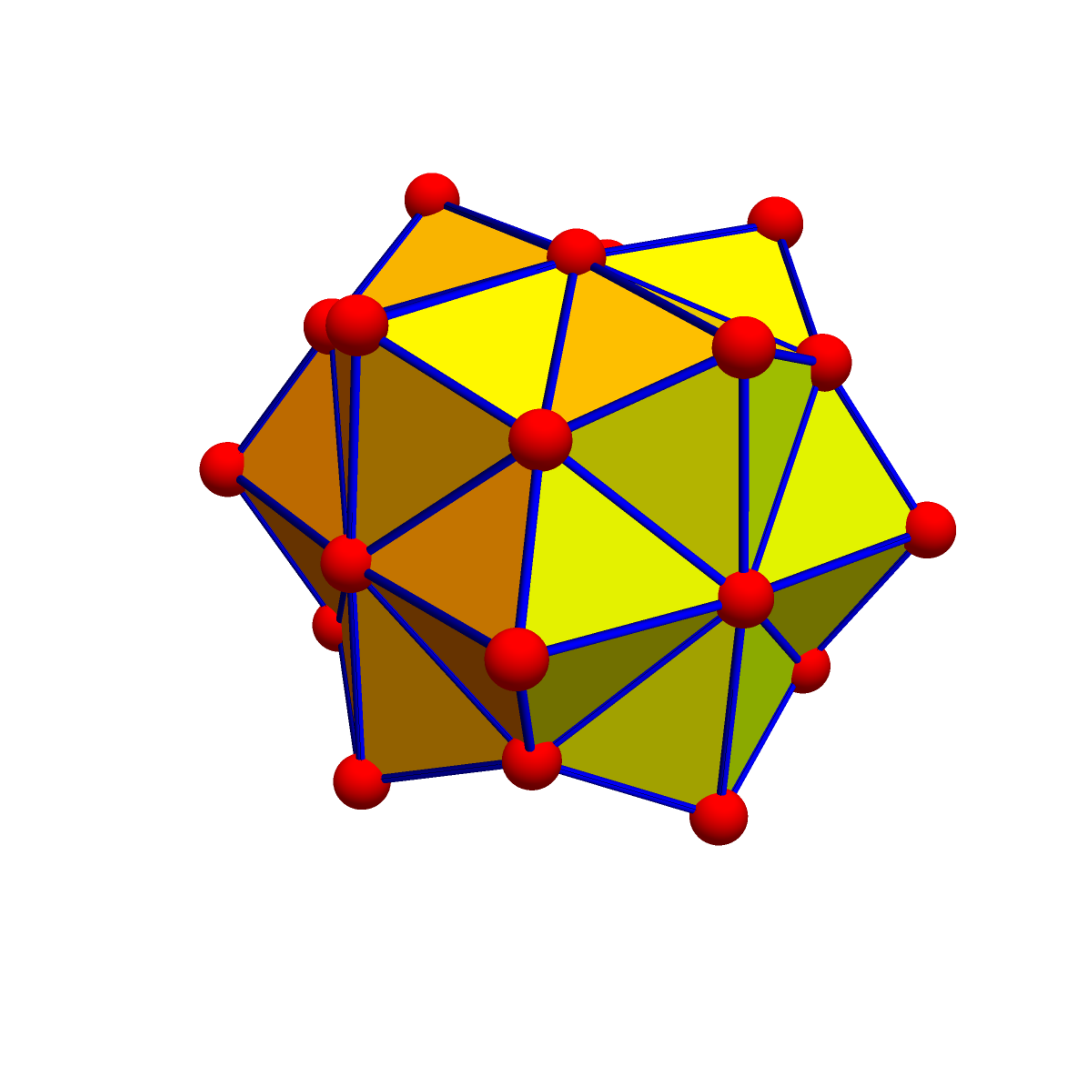}}
\caption{
The octahedron complex $G$ has $26 =6+12+8$ faces. The connection graph of $G$
is the Barycentric refinement of $G$. It is seen to the right. Is this the
complex which minimizes the trace of the Green function among all complexes with 
26 faces? 
}
\end{figure}


\begin{figure}[!htpb]
\scalebox{0.4}{\includegraphics{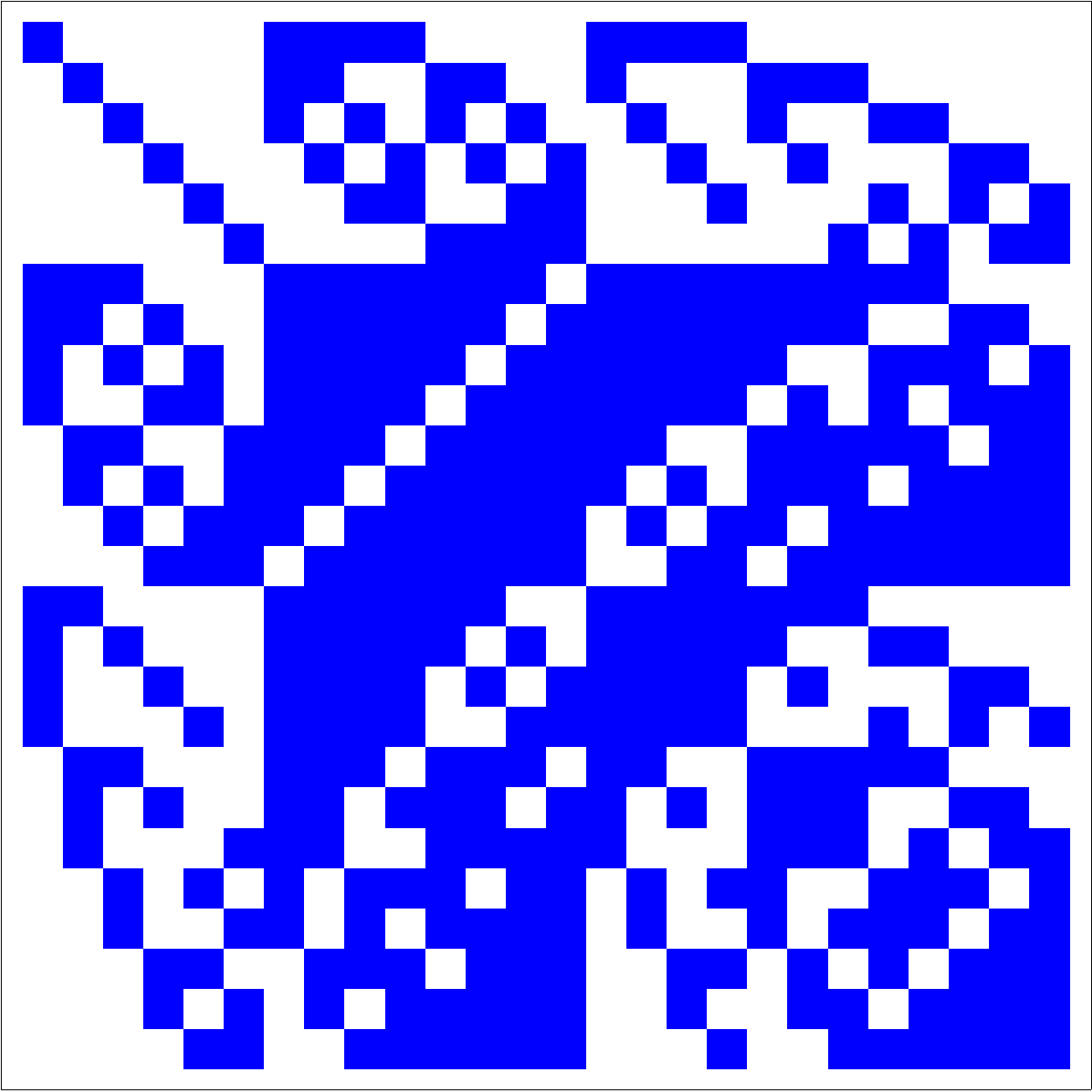}}
\scalebox{0.4}{\includegraphics{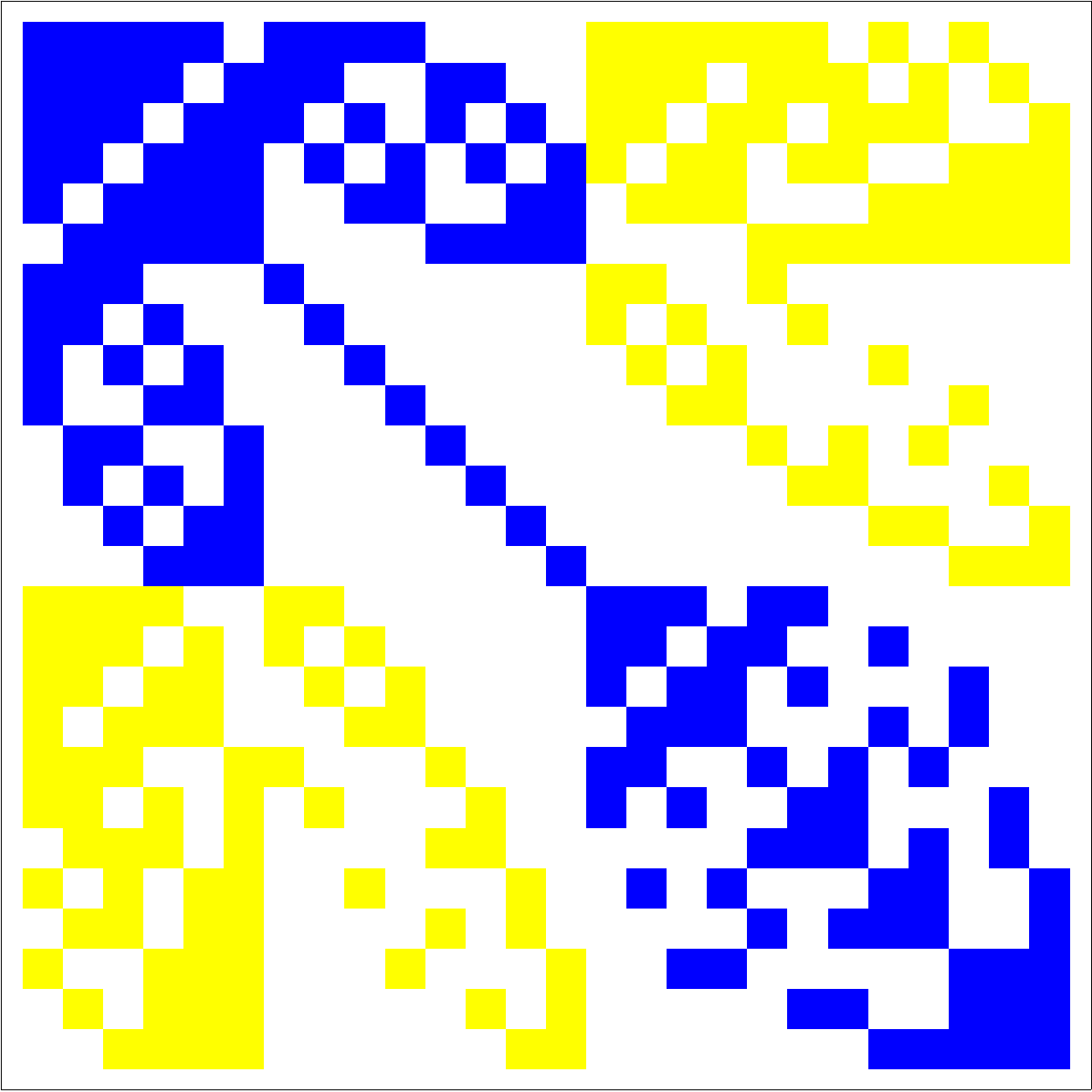}}
\caption{
\label{octahedronmatrix}
The $26 \times 26$ matrix $L$ seen to the left is the connection Laplacian of the
octahedron graph $G$. We have $L(x,y)=1$ if two complete subgraphs of $G$ intersect
and $L(x,y)=0$ else. The inverse $g$, the Green function is seen to the right. Also 
$g$ has $1$ in the diagonal as in the Barycentric refinement $G_1$ of $G$ every unit
sphere $S(x)$ is a circular graph. The complex $G$ is a discrete manifold. We don't
know of any other finite simplicial complex with $26$ elements which has a larger
trace of the Green function. 
}
\end{figure}

\paragraph{}
Finally, we have seen that the simplex generating function $f(x)$ and
its anti derivative $F$ can be used to compute Euler characteristic $\chi$, 
the Euler curvature and functional $\eta$ in a similar way
$$ \chi(G) = f(0)-f(-1), K = F(0)-F(-1),    \eta(G) = f'(0)-f'(-1) \; . $$
This of course prompts the question whether other similar function values of $f$ are 
geometrically interesting. The two functionals were now linked also 
by the Gauss-Bonnet relation $\eta(G) = \sum_x \chi(S(x))$ as well as the
old Gauss-Bonnet $\chi(G) = \sum_x K(x)$. 

\paragraph{}
There are other algebraic-analytic 
relations like ${\rm tr}(g) = f'(-1)$ 
and as a consequence of unimodularity, ${\rm det}(g) = (-1)^{1+(f_0(-1)+f_0(1))/2}$
but where $f_0$ is the $f$-vector generating function of $G$ itself, not of 
the Barycentric refinement $G_1$. The reason is that ${\rm det}(L)={\rm det}(g)$ 
is what we called the Fermi characteristic $\phi(G) = \prod_x (-1)^{{\rm dim}(x))}$ which is
$1$ if there are an even number of odd dimensional simplices present in $G$ and 
$-1$ if that number is odd. This number can change under Barycentric refinement
unlike the Euler characteristic $\chi(G) = \sum_x (-1)^{{\rm dim}(x)}$, where 
$\chi(G)=f_0(0)-f_0(-1) = f(0)-f(-1)=\chi(G_1)$ is the same for the simplex generating 
function of $G$ and $G_1$. 

\pagebreak

\section{Code}

As usual, the following Mathematica procedures can be copy-pasted from the ArXiv'ed LaTeX source
file to this document. Together with the text, it should be pretty clear what each procedure does. 

\begin{tiny}
\lstset{language=Mathematica} \lstset{frameround=fttt}
\begin{lstlisting}[frame=single]
UnitSphere[s_,a_]:=Module[{b=NeighborhoodGraph[s,a]},
  If[Length[VertexList[b]]<2,Graph[{}],VertexDelete[b,a]]];
UnitSpheres[s_]:=Map[Function[x,UnitSphere[s,x]],VertexList[s]];
F[A_,z_]:=A-z IdentityMatrix[Length[A]]; F[A_]:=F[A,-1];
FredholmDet[s_]:=Det[F[AdjacencyMatrix[s]]];
BowenLanford[s_,z_]:=Det[F[AdjacencyMatrix[s],z]];
CliqueNumber[s_]:=Length[First[FindClique[s]]];
ListCliques[s_,k_]:=Module[{n,t,m,u,r,V,W,U,l={},L},L=Length;
  VL=VertexList;EL=EdgeList;V=VL[s];W=EL[s]; m=L[W]; n=L[V];
  r=Subsets[V,{k,k}];U=Table[{W[[j,1]],W[[j,2]]},{j,L[W]}];
  If[k==1,l=V,If[k==2,l=U,Do[t=Subgraph[s,r[[j]]];
  If[L[EL[t]]==k(k-1)/2,l=Append[l,VL[t]]],{j,L[r]}]]];l];
Whitney[s_]:=Module[{F,a,u,v,d,V,LC,L=Length},V=VertexList[s];
  d=If[L[V]==0,-1,CliqueNumber[s]];LC=ListCliques;
  If[d>=0,a[x_]:=Table[{x[[k]]},{k,L[x]}];
  F[t_,l_]:=If[l==1,a[LC[t,1]],If[l==0,{},LC[t,l]]];
  u=Delete[Union[Table[F[s,l],{l,0,d}]],1]; v={};
  Do[Do[v=Append[v,u[[m,l]]],{l,L[u[[m]]]}],{m,L[u]}],v={}];v];
Barycentric[s_]:=Module[{v={},c=Whitney[s]},Do[Do[If[c[[k]]!=c[[l]] 
  && (SubsetQ[c[[k]],c[[l]]] || SubsetQ[c[[l]],c[[k]]]),
  v=Append[v,k->l]],{l,k+1,Length[c]}],{k,Length[c]}];
  UndirectedGraph[Graph[v]]];
ConnectionGraph[s_] := Module[{c=Whitney[s],n,A},n=Length[c];
  A=Table[1,{n},{n}];Do[If[DisjointQ[c[[k]],c[[l]]]||
  c[[k]]==c[[l]],A[[k,l]]=0],{k,n},{l,n}];AdjacencyGraph[A]];
ConnectionLaplacian[s_]:=F[AdjacencyMatrix[ConnectionGraph[s]]];
FredholmCharacteristic[s_]:=Det[ConnectionLaplacian[s]];
GreenF[s_]:=Inverse[ConnectionLaplacian[s]];
Energy[s_]:=Total[Flatten[GreenF[s]]];
BarycentricOp[n_]:=Table[StirlingS2[j,i]i!,{i,n+1},{j,n+1}];
Fvector[s_] := Delete[BinCounts[Length /@ Whitney[s]], 1]; 
Fvector1[s_]:=Module[{f=Fvector[s]},BarycentricOp[Length[f]-1].f];
Fv=Fvector; Fv1=Fvector1;
GFunction[s_,x_]:=Module[{f=Fv[s]},1+Sum[f[[k]]x^k,{k,Length[f]}]];
GFunction1[s_,x_]:=Module[{f=Fv1[s]},1+Sum[f[[k]]x^k,{k,Length[f]}]];
dim[x_]:=Length[x]-1; Pro=Product; W=Whitney; 
Euler[s_]:=Module[{w=W[s],n},n=Length[w];Sum[(-1)^dim[w[[k]]],{k,n}]];
Fermi[s_]:=Module[{w=W[s],n},n=Length[w];Pro[(-1)^dim[w[[k]]],{k,n}]];
Eta[s_]:=Tr[ConnectionLaplacian[s]-GreenF[s]]; 
Eta0[s_]:=Total[Map[Euler,UnitSpheres[s]]];
Eta1[s_]:=Total[Map[Euler,UnitSpheres[Barycentric[s]]]];
EtaG[s_]:=Module[{g=GFunction1[s,x]}, f[y_]:=g /. x->y;f'[0]-f'[-1]];
EulerG[s_]:=Module[{g=GFunction[s,x]},f[y_]:=g /. x->y;f[0] -f[-1] ];

s=RandomGraph[{23,60}]; sc = ConnectionGraph[s]; 
{Euler[s],Energy[s],EulerG[s]}
{Fermi[s],BowenLanford[sc,-1],FredholmCharacteristic[s]}
{Eta1[s],Eta[s],EtaG[s]} 
\end{lstlisting}
\end{tiny}

\pagebreak

\section{Examples}

\paragraph{}
Let $G$ be the house graph. It has the $f$-vector $(5,6,1)$, the generating function $f(x)=1+5x+6x^2+x^2$ and
the Euler characteristic $5-6+1=0$. The unit spheres have all Euler characteristic $\chi(S(x))=2$ except the 
roof-top where $\chi(S(x))=1$.
We therefore have $\eta_0(G)=9$. The Barycentric refinement $G_1$ of $G$ has the $f$-vector $(12,18,6)$. 
The unit sphere Euler characteristic spectrum is $(0, 2, 2, 1, 1, 1, 1, 2, 2, 2, 2, 2)$ totals to
$\eta_1(G)=\eta_0(G_1)=18$. The connection graph $G'$ of $G$ already has dimension $4$. Its 
$f$-vector is $(12, 29, 27, 12, 2)$ and $\chi(G')=0$ still. 

\begin{figure}[!htpb]
\scalebox{1.1}{\includegraphics{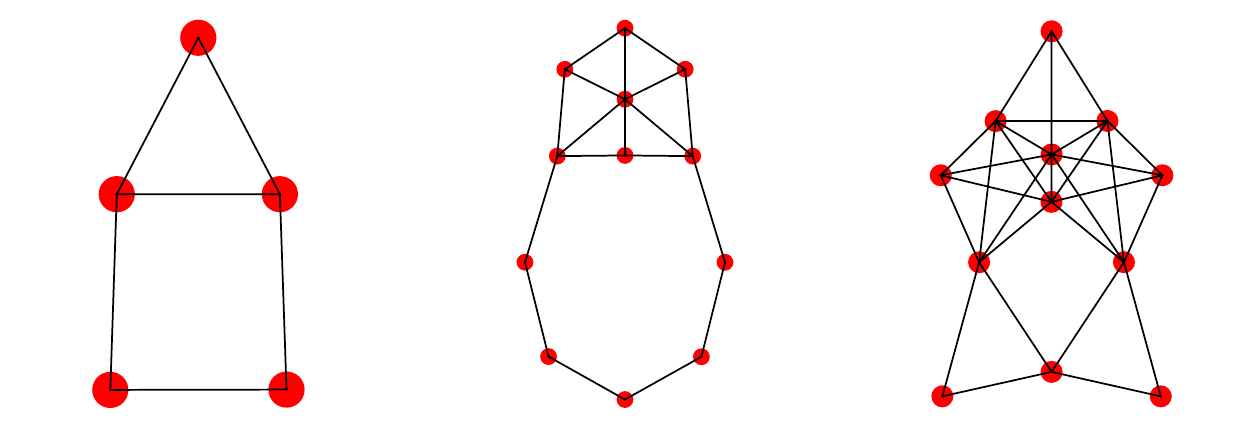}}
\caption{
The house graph $G$, its barycentric refinement $G_1$ and
its connection graph $G'$. 
}
\end{figure}

The connection Laplacian $L$ and its inverse $g$ are
\begin{tiny} $$
\arraycolsep=2pt\def\arraystretch{1.0}
L=\left[ \begin{array}{cccccccccccc}
1&0&1&1&0&1&0&1&1&1&1&1\\
0&1&0&0&0&0&1&1&0&0&0&0\\
1&0&1&0&0&0&0&1&1&1&0&0\\
1&0&0&1&0&0&0&0&1&0&1&1\\
0&0&0&0&1&0&1&0&0&0&1&0\\
1&0&0&0&0&1&0&0&0&1&0&1\\
0&1&0&0&1&0&1&1&0&0&1&0\\
1&1&1&0&0&0&1&1&1&1&0&0\\
1&0&1&1&0&0&0&1&1&1&1&1\\
1&0&1&0&0&1&0&1&1&1&0&1\\
1&0&0&1&1&0&1&0&1&0&1&1\\
1&0&0&1&0&1&0&0&1&1&1&1\\
\end{array} \right],
g=\left[ \begin{array}{cccccccccccc}
1&0&1&1&0&1&0&0&-1&-1&0&-1\\
0&-1&-1&0&-1&0&1&1&0&0&0&0\\
1&-1&-1&0&0&0&0&1&0&0&0&-1\\
1&0&0&-1&-1&0&0&0&0&-1&1&0\\
0&-1&0&-1&-1&0&1&0&0&0&1&0\\
1&0&0&0&0&0&0&0&-1&0&0&0\\
0&1&0&0&1&0&-1&0&0&0&0&0\\
0&1&1&0&0&0&0&-1&0&0&0&0\\
-1&0&0&0&0&-1&0&0&0&1&0&1\\
-1&0&0&-1&0&0&0&0&1&0&0&1\\
0&0&0&1&1&0&0&0&0&0&-1&0\\
-1&0&-1&0&0&0&0&0&1&1&0&0\\
\end{array} \right] \; . $$
\end{tiny}
The trace of $L$ is $12$, the trace of $g$ is $-6$. 
The super trace of both $L$ and $g$ or the sum $\sum_{x,y} g(x,y)$ are
all $\chi(G) =0$. The spectrum of $L$ is
$\sigma(L)= \{$ $-1.30009$, $-0.827091$, $-0.646217$, $-0.528497$, $-0.338261$, $-0.255285$, 
$0.245226$, $1.20906$, $1.72111$, $2.9563$, $3.17017$, $6.59358$ $\}$. We see in
most random graphs that about half of the eigenvalues are negative and that the negative 
spectrum has smaller amplitude. 

\paragraph{}
The {\bf double pyramid} $G$ is a $2$-variety with $7$ vertices. It can be obtained
by making two separate pyramid construction over a wheel graph. One can also write it
as the Zykov join $P_3+C_4$ of $P_3$ with $C_4$ or then $P_3+P_2+P_2=P_3 + 2 P_2$. 
While the octahedron $O=P_2+C_4=P_2+P_2+P_2=3P_2$ has has 
$\eta(O)=0$ and all unit spheres with Euler characteristic $0$,
now there are $4$ vertices in $G$ where $S(x)$ has Euler characteristic $-1$. 
The graph $G$ has $f$-vector $(7,16,12)$ and Euler characteristic $3$ and Betti
vector $(1,0,2)$. The $f$-vector of the Barycentric refinement is $(35,104,72)$. 
We have $\eta_0(G)=-4$. The Barycentric refinement has $8$ vertices with Euler
characteristic $-1$ and $\eta_1(G)=-8$. 

\begin{figure}[!htpb]
\scalebox{1.1}{\includegraphics{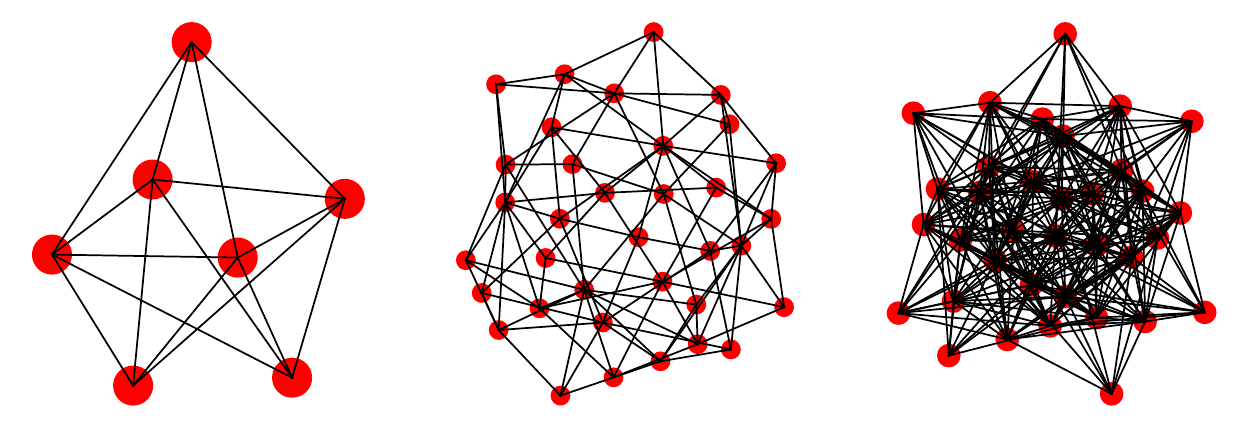}}
\caption{
The double pyramid $G = P_3 + C_4=P_3+2 P_2$ is a triple suspension of
a circle. Compare this with other Zykov sums: 
$P_1+C_4=W_4$, the wheel graph or $P_2 +C_4=O$, the octahedron. 
We see also its Barycentric refinement $G_1$ and its connection graph $G'$.
}
\end{figure}

The connection Laplacian 
\begin{tiny} $$
\arraycolsep=2pt\def\arraystretch{1.0}
L=\left[ \begin{array}{ccccccccccccccccccccccccccccccccccc}
1&0&0&0&0&0&0&1&1&1&1&0&0&0&0&0&0&0&0&1&1&1&1&0&0&0&0&0&0&0&0&0&0&0&0\\ 
0&1&0&0&0&0&0&1&1&0&0&1&1&1&1&0&0&0&0&1&0&0&0&1&1&1&1&0&0&0&0&0&0&0&0\\ 
0&0&1&0&0&0&0&1&0&1&0&1&1&0&0&1&1&0&0&0&1&0&0&1&0&0&0&1&1&1&0&0&0&0&0\\ 
0&0&0&1&0&0&0&0&1&0&1&0&0&1&1&0&0&1&1&0&0&1&0&0&1&0&0&0&0&0&1&1&1&0&0\\ 
0&0&0&0&1&0&0&0&0&1&1&0&0&0&0&1&1&1&1&0&0&0&1&0&0&0&0&1&0&0&1&0&0&1&1\\ 
0&0&0&0&0&1&0&0&0&0&0&1&0&1&0&1&0&1&0&0&0&0&0&0&0&1&0&0&1&0&0&1&0&1&0\\ 
0&0&0&0&0&0&1&0&0&0&0&0&1&0&1&0&1&0&1&0&0&0&0&0&0&0&1&0&0&1&0&0&1&0&1\\ 
1&1&1&0&0&0&0&1&1&1&1&1&1&1&1&1&1&0&0&1&1&1&1&1&1&1&1&1&1&1&0&0&0&0&0\\ 
1&1&0&1&0&0&0&1&1&1&1&1&1&1&1&0&0&1&1&1&1&1&1&1&1&1&1&0&0&0&1&1&1&0&0\\ 
1&0&1&0&1&0&0&1&1&1&1&1&1&0&0&1&1&1&1&1&1&1&1&1&0&0&0&1&1&1&1&0&0&1&1\\ 
1&0&0&1&1&0&0&1&1&1&1&0&0&1&1&1&1&1&1&1&1&1&1&0&1&0&0&1&0&0&1&1&1&1&1\\ 
0&1&1&0&0&1&0&1&1&1&0&1&1&1&1&1&1&1&0&1&1&0&0&1&1&1&1&1&1&1&0&1&0&1&0\\ 
0&1&1&0&0&0&1&1&1&1&0&1&1&1&1&1&1&0&1&1&1&0&0&1&1&1&1&1&1&1&0&0&1&0&1\\ 
0&1&0&1&0&1&0&1&1&0&1&1&1&1&1&1&0&1&1&1&0&1&0&1&1&1&1&0&1&0&1&1&1&1&0\\ 
0&1&0&1&0&0&1&1&1&0&1&1&1&1&1&0&1&1&1&1&0&1&0&1&1&1&1&0&0&1&1&1&1&0&1\\ 
0&0&1&0&1&1&0&1&0&1&1&1&1&1&0&1&1&1&1&0&1&0&1&1&0&1&0&1&1&1&1&1&0&1&1\\ 
0&0&1&0&1&0&1&1&0&1&1&1&1&0&1&1&1&1&1&0&1&0&1&1&0&0&1&1&1&1&1&0&1&1&1\\ 
0&0&0&1&1&1&0&0&1&1&1&1&0&1&1&1&1&1&1&0&0&1&1&0&1&1&0&1&1&0&1&1&1&1&1\\ 
0&0&0&1&1&0&1&0&1&1&1&0&1&1&1&1&1&1&1&0&0&1&1&0&1&0&1&1&0&1&1&1&1&1&1\\ 
1&1&0&0&0&0&0&1&1&1&1&1&1&1&1&0&0&0&0&1&1&1&1&1&1&1&1&0&0&0&0&0&0&0&0\\ 
1&0&1&0&0&0&0&1&1&1&1&1&1&0&0&1&1&0&0&1&1&1&1&1&0&0&0&1&1&1&0&0&0&0&0\\ 
1&0&0&1&0&0&0&1&1&1&1&0&0&1&1&0&0&1&1&1&1&1&1&0&1&0&0&0&0&0&1&1&1&0&0\\ 
1&0&0&0&1&0&0&1&1&1&1&0&0&0&0&1&1&1&1&1&1&1&1&0&0&0&0&1&0&0&1&0&0&1&1\\ 
0&1&1&0&0&0&0&1&1&1&0&1&1&1&1&1&1&0&0&1&1&0&0&1&1&1&1&1&1&1&0&0&0&0&0\\ 
0&1&0&1&0&0&0&1&1&0&1&1&1&1&1&0&0&1&1&1&0&1&0&1&1&1&1&0&0&0&1&1&1&0&0\\ 
0&1&0&0&0&1&0&1&1&0&0&1&1&1&1&1&0&1&0&1&0&0&0&1&1&1&1&0&1&0&0&1&0&1&0\\ 
0&1&0&0&0&0&1&1&1&0&0&1&1&1&1&0&1&0&1&1&0&0&0&1&1&1&1&0&0&1&0&0&1&0&1\\ 
0&0&1&0&1&0&0&1&0&1&1&1&1&0&0&1&1&1&1&0&1&0&1&1&0&0&0&1&1&1&1&0&0&1&1\\ 
0&0&1&0&0&1&0&1&0&1&0&1&1&1&0&1&1&1&0&0&1&0&0&1&0&1&0&1&1&1&0&1&0&1&0\\ 
0&0&1&0&0&0&1&1&0&1&0&1&1&0&1&1&1&0&1&0&1&0&0&1&0&0&1&1&1&1&0&0&1&0&1\\ 
0&0&0&1&1&0&0&0&1&1&1&0&0&1&1&1&1&1&1&0&0&1&1&0&1&0&0&1&0&0&1&1&1&1&1\\ 
0&0&0&1&0&1&0&0&1&0&1&1&0&1&1&1&0&1&1&0&0&1&0&0&1&1&0&0&1&0&1&1&1&1&0\\ 
0&0&0&1&0&0&1&0&1&0&1&0&1&1&1&0&1&1&1&0&0&1&0&0&1&0&1&0&0&1&1&1&1&0&1\\ 
0&0&0&0&1&1&0&0&0&1&1&1&0&1&0&1&1&1&1&0&0&0&1&0&0&1&0&1&1&0&1&1&0&1&1\\ 
0&0&0&0&1&0&1&0&0&1&1&0&1&0&1&1&1&1&1&0&0&0&1&0&0&0&1&1&0&1&1&0&1&1&1\\ 
\end{array} \right] $$
\end{tiny}
has $659$ entries $1$ which is $|V(G')|+2|E(G')|=35+2 \cdot 312$. The 
Green function $g=L^{-1}$ has entries $\{-2,-1,0,1,2\}$. We have ${\rm tr}(g)=43$
and ${\rm tr}(L)=35$ and $\eta(G)=35-43=-8$. The spectrum of $L$ has the 
convex hull $[-3.30278,20.0327]$. There are 16 negative and 19 positive
eigenvalues and 16 negative eigenvalues. The spectrum of $g$ has
the convex hull $[-3.30278,16]$. \\

When making a triple suspension $G=C_n + P_3$ over a larger circle we get a graph with 
$\eta_0(G)=-n$ and $\eta(G)=-2n$. By punching two small holes at vertices of degree 4 
into $G_1$, the graph can be rendered to be contractible with $\eta(G)=8-2n$. 
Attaching this using using wedge sum to an other graph illustrates
again that we can lower $\eta$ arbitrarily in any homotopy class of graphs. 

\bibliographystyle{plain}

\end{document}